\documentclass[a4paper,12pt]{article}
\usepackage[applemac]{inputenc}
\usepackage{amssymb,mathrsfs,amsthm}
\usepackage{amsmath}
\usepackage{stmaryrd} 

\usepackage{times}

\usepackage{geometry}
\usepackage{empheq}
\usepackage{enumerate,enumitem}
\usepackage{dsfont}
\usepackage{tikz}
\usepackage{float,caption}
\usepackage{changepage}
\usepackage[hidelinks]{hyperref}

\newtheoremstyle{rec}
{\topsep}
{\topsep}
{\itshape}
{}
{\itshape}
{.}
{0.5em}
{\thmname{#1}\thmnumber{ (#2)}\thmnote{~: \textit{#3}}}%

\theoremstyle{rec}
\newtheorem*{Hk}{$(H_k)$}

\theoremstyle{plain}
\newtheorem{lemma}{Lemma}[section]
\newtheorem{proposition}{Proposition}[section]
\newtheorem{theorem}{Theorem}[section]
\newtheorem{corollary}{Corollary}[section]

\theoremstyle{remark}
\newtheorem{remark}{Remark}[section]
\newtheorem{hyp}{Condition}[section]

\newtheorem{closing}{Closing remark}[section]

\theoremstyle{definition}

\newtheorem*{assumptions}{Assumptions on $p $}
\newtheorem{definition}{Definition}[section]
\newtheorem{claim}{Claim}[section]
\newtheorem*{comparison}{Comparison principle for}
\newtheorem*{caseone}{Case one: $t_{\Sigma}(h_1)\geq t_{\Sigma}(h_0)$}
\newtheorem*{casetwo}{Case two: $t_{\Sigma}(h_1)< t_{\Sigma}(h_0)$}

\makeatletter
\renewcommand\theequation{\thesection.\arabic{equation}}
\@addtoreset{equation}{section}
\makeatother

\renewcommand{\theequation}{\arabic{section}.\arabic{equation}}

\title{Finite-time singularity of the stochastic harmonic map flow\thanks{Second version, November 2018.}}
\author{Antoine Hocquet\thanks{Email: antoine.hocquet@wanadoo.fr, Tel.: +49 30 314-25728, Fax: +49 30 314-25191}
  \thanks{This work was done while the author was affiliated at the CMAP, \'Ecole Polytechnique, CNRS, Universit\'e Paris Saclay, 91128, Palaiseau, France.}
\\{\small Technische Universit\"at Berlin, Instit\"ut f\"ur Mathematik, Strasse des 17. Juni 136,}
\\{\small Berlin-Charlottenburg, Germany.}
}
\date{}

\newcommand{\R}{\ensuremath{\mathbb{R}}}
\newcommand{\D}{\ensuremath{{\mathbb{D}^2}}}
\newcommand{\E}{\ensuremath{\mathbb{E}}}
\newcommand{\M}{\ensuremath{M}}
\newcommand{\N}{\ensuremath{\mathbb{N}}}
\renewcommand{\H}{\ensuremath{{\mathscr{H}}}}
\newcommand{\HH}{\ensuremath{{\mathscr{\bar H}}}}
\newcommand{\F}{\ensuremath{{\mathscr{F}}}}
\renewcommand{\P}{\ensuremath{\mathbb{P}}}
\newcommand{\J}{\ensuremath{J}}
\newcommand{\I}{\ensuremath{I}}

\newcommand{\spt}{\ensuremath{\mathrm{Supp}\,}}

\DeclareMathOperator*{\T}{\,{}^t}
\DeclareMathOperator*{\interior}{Int}
\renewcommand{\d}{\ensuremath{\hspace{0.05em}\mathrm{d}}}        

\providecommand{\msc}[1]{{\small \textit{Mathematics Subject Classification---} #1}}
\providecommand{\keywords}[1]{{\small \textit{Keywords---} #1}}

\begin{document}
\maketitle


\msc{60H15 (35R60), 58E20, 35K55, 35B44}

\begin{abstract}
We investigate the influence of an infinite dimensional Gaussian noise on the bubbling phenomenon for the stochastic harmonic map flow $u(t,\cdot ):{\mathbb D}^2\to\mathbb S^2$, from the two-dimensional unit disc onto the sphere. The diffusion term is assumed to have range one pointwisely in the tangent space $T_{u(t,x)}\mathbb S^2$, so that the noise preserves the 1-corotational symmetry of solutions. Under the assumption that its space-correlation is of trace class (in some appropriate Hilbert space), we prove that the noise generates blow-up with positive probability. This scenario happens no matter how we choose the initial data, provided it fulfills the latter symmetry assumption.
\end{abstract}

\tableofcontents

\section{Introduction and main results}\label{section_introduction}
\label{intro}
\subsection{Motivations}
The effect of a noise term on the appearance of a finite-time singularity has already been investigated for several stochastic PDE's, including the Schr\"odinger equation \cite{de2002effect,de2005blow} where it is shown to generate blow-up with positive probability, for any initial data.
Some results in the same spirit have been obtained for the stochastic heat equation \cite{mueller1993blowup,mueller2000critical,dozzi2010finite}, and also for the so-called Dyadic Model \cite{romito2014uniqueness}, where the author shows in addition the ineluctability of the blow-up.
Our work comes from an attempt to understand the effect of noise on the bubbling phenomenon for the two-dimensional stochastic Landau-Lifshitz-Gilbert equation (SLLG) \cite{landau1935theory,gilbert1955lagrangian}, for which the stochastic harmonic map flow turns out to be a simplified version.

The magnetization of a ferromagnetic material $\M\subset\R^3$ can be represented as a time-dependent continuum $u:[0,T]\times \M\to\mathbb S^2$, whose stationary solutions should solve the minimization problem for the Brown energy $E\equiv(1/2)\int_\M|\nabla u|^2\d\M$, under the pointwise constraint 
\begin{equation}\label{constraint}
|u(t,x)|=1\,, \enskip \text{a.e.}
\end{equation}
The latter energy corresponds to closest neighbour interaction, and here we do not take into account other contributions such as anisotropy, stray field or external field (for details see \cite{brown1963micromagnetics}).
Ignoring for now the noise and dropping the so-called gyromagnetic term (which has no effect on the energy), we obtain the harmonic map flow from $\M$ to the unit sphere $\mathbb S^2$, namely
\begin{equation}\label{HMF}\tag{HMF}
\left\{
\begin{aligned}
&\partial_t u=\Delta u+u|\nabla u|^2\qquad \text{on}\enskip (0,T]\times \M\enskip ,
\\
& u=\varphi\,\qquad\text{on}\enskip [0,T]\times\partial \M\cup\{0\}\times \M\,,
\end{aligned}
\right.
\end{equation}
where $\Delta $ denotes the Laplacian with respect to each of the components of $u\equiv(u^1,u^2,u^3)$,
and $|\nabla u|^2:=\sum_{i\leq 3,j\leq2}(\partial _ju^i)^2.$ 
Note that \eqref{HMF} is in fact the gradient flow associated to $E$, under the constraint \eqref{constraint}.
This model has been independently studied e.g.\ in \cite{eells1964harmonic,hamilton1975harmonic,eells1978report,eells1988another}, where target manifolds more general than $\mathbb S^2$ are considered.
It provides a tool to construct a \textit{harmonic map} in the homotopy class of $\varphi $, namely a regular solution to the minimization problem associated to $E$.

Existence of finite-time blowing-up solutions has been shown in dimensions strictly greater than 2 in \cite{coron1989equations}, and then later in 2D in \cite{chang1992finite}.
The two-dimensional case is more challenging in the sense that
the $H^1$ a priori estimate barely fails to give well-posedness.
Another specific feature of the 2D case is that
in the absence of noise, there can be at most one energy-decreasing solution in the latter class \cite{freire1995uniqueness}. Note however that the energy cannot be decreasing with a stochastic term as in \eqref{SHMF} below (as can be shown by applying It\^o Formula to $E$).

\paragraph{Singularities of symmetric solutions in 2D.}
The case where $\M$ is a surface is energy-critical, meaning that blow-up by concentration of energy can occur. If the initial energy is less than some quantity
\begin{equation}\label{quantum}
 E(\varphi )<\epsilon _1\,,
\end{equation}
depending on $\M$ only, then 
the solution $u(t,\cdot )$ of
\eqref{HMF}
is global and uniformly converges towards an harmonic map  $u_\infty$ as $t\to\infty$
(see \cite{eells1964harmonic,struwe1985evolution,kung1989heat}).

Oppositely, the local solution $u$ of \eqref{HMF} may not be defined globally (in the classical sense) if \eqref{quantum} is not fulfilled.
Finite-time blowing-up solutions were provided in \cite{chang1992finite},
for the case $u:[0,T]\times\D\equiv\{x\in\R^2:|x|<1\}\to\mathbb S^2$.
Considering 1-corotational solutions\footnote{In the existing literature, these maps are often called ``equivariant'', or 1-equivariant, although the latter can have by definition an additional degree of freedom $b$, so that $u(r\cos\theta ,r\cos\theta )=R_\theta \T(a(r),b(r),c(r))$ with $a^2+b^2+c^2\equiv1$, $R_\theta $ corresponding to the rotation of angle $\theta $ and axis $\mathbf k$.
The form given above corresponds to the special case where $a(r)=\sin h(r),b(r)=0,c(r)=\cos h(r)$ and should be rather called ``1-corotational'' (see for instance \cite{van2013stability}).
}
of the form $u=u_h$ with
\begin{equation}\label{1-corotational}
u_h(t,x):=\left(\frac{x}{|x|}\sin h(t,|x|);\cos h(t,|x|)\right)\,.
\end{equation}
the system \eqref{HMF} is reduced to a parabolic equation on the scalar map $h(t,r)$:
\begin{equation}\label{eq:h_classical}
\left\{\begin{aligned}
&\partial_t h=\partial_{rr}h+\dfrac{\partial_rh}{r}-\dfrac{\sin2h}{2r^2}\qquad\text{for}\enskip(t,r)\in  [0,T]\times(0,1)\,,\\
&h(t,0)=0 \enskip ,\enskip h(t,1)=\gamma\qquad\enskip\text{for}\enskip t\in[0,T]\,,\\
&h(0,r)=h_0(r)\enskip\qquad\qquad\enskip\text{for}\enskip r\in(0,1)\,,
\end{aligned}\right.
\end{equation}
and a comparison principle for \eqref{eq:h_classical} holds.
In \cite{chang1992finite}, the authors exhibit a class of self-similar, blowing-up subsolutions for the parabolic problem \eqref{eq:h_classical}, implying by comparison:
\begin{equation}\label{blowup_h}
\partial _rh(t,0)\underset{t\to t_*}{\longrightarrow}\infty\,,
\end{equation}
for some $t_*>0$ depending on the initial data.
As described by M. Struwe \cite{struwe1985evolution}, this implies ``forward bubbling'' for $u_h$, namely: as $t\nearrow t_*$, the energy concentrates at the center of $\D$. The solution can be extended in distributional sense after $t_*$ by simply taking the weak limit in $H^1$.

In \eqref{eq:h_classical}, the number $\gamma $ is the angle between $\varphi |_{\partial \D}$ and the vertical axis, so that setting $\gamma :=0$ corresponds to
\begin{equation}\label{boundary_value_k}
{\varphi}|_{\partial \D}=\mathbf k:=(0,0,1)\,.
\end{equation}
We will work in this homogeneous setting, although in \cite{chang1992finite} the authors assume $\gamma >\pi$ (blow-up for $\gamma =0$ is actually shown in 
\cite{bertsch2002nonuniqueness}).

\paragraph{Stability of blow-up under random perturbations.}
Concerning the deterministic equation \eqref{HMF}, stability results (under perturbation of the initial data) have been obtained in \cite{merle2011blow,raphael2013stable,van2013stability}.
In \cite{merle2011blow}, the authors show the existence, but instability, of initial data leading to blow-up for the Landau-Lifshitz-Gilbert equation (more precisely for $\partial _tu=u\times\Delta u$).
It appears that instability is due to the necessary extra degree of freedom compared to the ``overdamped model'' \eqref{HMF}, for which
a reduction to a scalar problem \eqref{eq:h_classical} is possible.
In this context, P.~Rapha\"el and R.~Schweyer (in case where $\D$ is replaced by $\R^2$) have shown that the pre-blow-up set is stable under small perturbations \textit{in the direction preserving 1-corotational symmetry}.

Whether \eqref{blowup_h} could be observed or not in presence of noise is the main topic treated in the present paper. Note that it is different -- though related -- from the stability results obtained above, for the noise modifies the dynamics, and not the initial data.
Although being a seemingly academic question, it echoes practical issues related to magnetic storage devices (see e.g.\ \cite{braun2000stochastic}), for which blow-up could be thought as spontaneous reversal of magnetization.
The stochastic term in (SLLG) corresponds to thermal fluctuations, see \cite{neel1946bases,brown1963thermal,Berkov}, which in theory are uncorrelated in time and space. A Gaussian white noise acting orthogonally to $u(t,x)$ (hence preserving \eqref{constraint}) has to be added in the equation, giving e.g.\ :
\begin{equation}\label{SHMF1}
\left\{
\begin{aligned}
&\d u=(\Delta u+u|\nabla u|^2)\d t+ u\times\circ \d W \,,
\qquad\text{on}\enskip (0,T]\times\D\,,
\\
&u= \mathbf k\,,
\hspace{7em}\text{on}\enskip \Sigma :=\{0\}\times\D\cup[0,T]\times\partial \D\,,&&
\end{aligned}\right.
\end{equation}
where $u\times$ denotes vector product, and
$t\in\R_+\mapsto W(t)\equiv(w_1(t),w_2(t),w_3(t))\in L^2(\D;\R^3)$ is a Wiener process (with a given covariance in space), whereas ``$\circ$'' means that the Stratonovitch rule is used.

Recent results for (SLLG) -- the equation obtained when the term $u\times\Delta u$ is added to the drift of \eqref{SHMF1} -- in dimensions less than or equal to 3, have been obtained in \cite{brzezniak2013weak,banas2013stochastic,banas2013convergent,GoldysLeTran,alouges2014semi}.
In these works a notion of ``weak martingale solution'' corresponding to that of \cite{DPZ} is used. In particular, solutions belong pathwisely to the space $C([0,T];L^2)\cap L^\infty([0,T];H^1).$ For a direct treatment of \eqref{SHMF1} on the two-dimensional torus, where strong solutions are obtained, we refer the reader to \cite{hocquet2018struwe}.
Strictly speaking, here we will focus on a model different from \eqref{SHMF1}, which is more adapted to our purpose, see \eqref{SHMF} below.

\paragraph{Another version of the Stochastic Harmonic Map Flow.}
It is worth noting that,
as observed in a note of G.C.\ Price and D.\ Williams dating from the 80s \cite{price1983rolling}, if $B\equiv(B_1,B_2,B_3)$ is a Brownian motion in $\R^3,$ then the three-dimensional SDE
\begin{equation}\label{BM_1}
\d X=X\times\circ \d B\,,
\qquad t\in(0,T]\,,
\qquad X_0\enskip \text{given in}\enskip \mathbb S^2\,,
\end{equation}
gives a description of the Brownian motion on the manifold $\mathbb{S}^2\subset\R^3,$ in the sense that its infinitesimal generator agrees with the Laplace-Beltrami operator.
Therefore, at an intuitive level,
the term ``$u\times\circ \d W$'' in \eqref{SHMF1} should be understood as a white noise with values in the ``tangent space of $u$ in $L^2(\D;\mathbb S^2)$'', a difficulty with that terminology being that the latter space admits no infinite-dimensional Riemannian manifold structure (see the related discussion in \cite[Chap.\ II\S6]{eells1964harmonic}).

As for the manifold $\mathbb S^2,$ it is seen from the formula given in 
\cite[Chapter IX, Theorem 1A]{elworthy1982stochastic}
that if we define the standard mobile frame associated to the spherical coordinates of $X\equiv( \cos \theta \sin\phi  , \sin \theta \sin \phi ,\cos \phi ),$ namely:
\begin{equation}\label{frame}
\left[
\begin{aligned}
&\Phi _{\theta ,\phi }:=(\cos \theta \cos \phi ,\sin \theta\cos \phi ,-\sin \phi )\,,\\
&\Theta _{\theta ,\phi }=(-\sin \theta,\cos \theta,0)\,,
\end{aligned}\right.
\end{equation}
then the Brownian motion on $\mathbb{S}^2$ can be locally described by the SDE
\begin{equation}\label{BM_2}
\d X=\Phi_{\theta ,\phi } \circ \d B_1 +\frac{\Theta_{\theta ,\phi }}{\sin\phi }\circ\d B_2\,,
\qquad t\in(0,T]\,,
\qquad X_0\in\mathbb{S}^2\setminus\{\mathbf k\}\,,
\end{equation}

If instead of \eqref{SHMF1} we consider an SDE describing the dynamics of a system of $N$ spins (replace the drift by minus the gradient of a suitable exchange energy), driven by $N$ independent Brownian motions $B^1,\dots ,B^N$ in $\mathbb R^3,$ it is straightforward to see that the law of the solution $u=(u_1,\dots,u_N):[0,T]\to\R^{3N}$ to \eqref{SHMF1} would remain unchanged if we replaced the terms $u_i\circ\times\d B^i$ by that given in \eqref{BM_2} (see also \cite[Prop.\ 3.2]{neklyudov2013role} for related computations).
By analogy with the finite-dimensional case,
given $g,h:[0,T]\times \D\to\R$ and letting (with a slight abuse of notation)
\begin{equation}\label{unknown_g_h}
u=( \cos g\sin h, \sin g\sin h,\cos h)\,,\quad 
\Phi _u:=\Phi_{g,h}\,,\quad
\Theta _u:=\Theta _{g,h}\,,
\end{equation}
then the equation 
\begin{equation}\label{SHMF}\tag{SHMF}
\left\{
\begin{aligned}
&\d u=(\Delta u+u|\nabla u|^2)\d t+\frac{ \Theta_u }{\sin h }\circ \d w_1+\Phi_u\circ \d w_2\,,
\\
&u=\mathbf \varphi \enskip \text{on}\enskip \Sigma\,,
\end{aligned}
\right.
\end{equation}
where as before $w_1$ and $w_2$ are Wiener processes with values in $L^2(\D;\mathbb R),$
appears to be ``more satisfactory'' in the sense that redundancies in the definition of the noise term are avoided (there are only two components in the noise term, not three). Note that in the present article, it is really \eqref{SHMF} that we are dealing with -- actually a degenerate version thereof -- and not \eqref{SHMF1}, although we guess that it should be possible to prove their equivalence in law.

The case when $w_1$ and $w_2$ are space-time white noises will not be treated in this article. On the one hand it is well-known that parabolic equations of the form \eqref{SHMF} can be ill-posed in dimension two, see e.g.\ \cite{hairer2012triviality}. 
On the other hand, we need the noise to be regular enough to guarantee that the notion of blow-up makes sense. 
We will build solutions that are strong in the probabilistic sense (see Definition \ref{def:weak_sol} below), and sufficiently regular in space, so that the singular time $\tau $ corresponds to the first moment when the solution leaves $C^1(\D)$. Further assumptions on the spatial correlation of $W$ will be done below.

We also point out that the existence of a finite-time blowing-up solution
is mixed with the question of the uniqueness of weak solutions (see \cite{bertsch2002nonuniqueness}),
but this problem will not be adressed here.
\subsection{Main results}
As immediately seen in \eqref{SHMF},
there is no hope to preserve 1-corotational symmetry along the flow if $w_1\neq0$, so that we simply drop this term and set $w:=w_2$. We also assume that $w(t,x)$ \textit{depends only on $t$ and $r\equiv|x|$}, see Fig.\ \ref{fig:explication_equivariant}.
\begin{figure}
\begin{center}
 \includegraphics[width=0.55\linewidth]{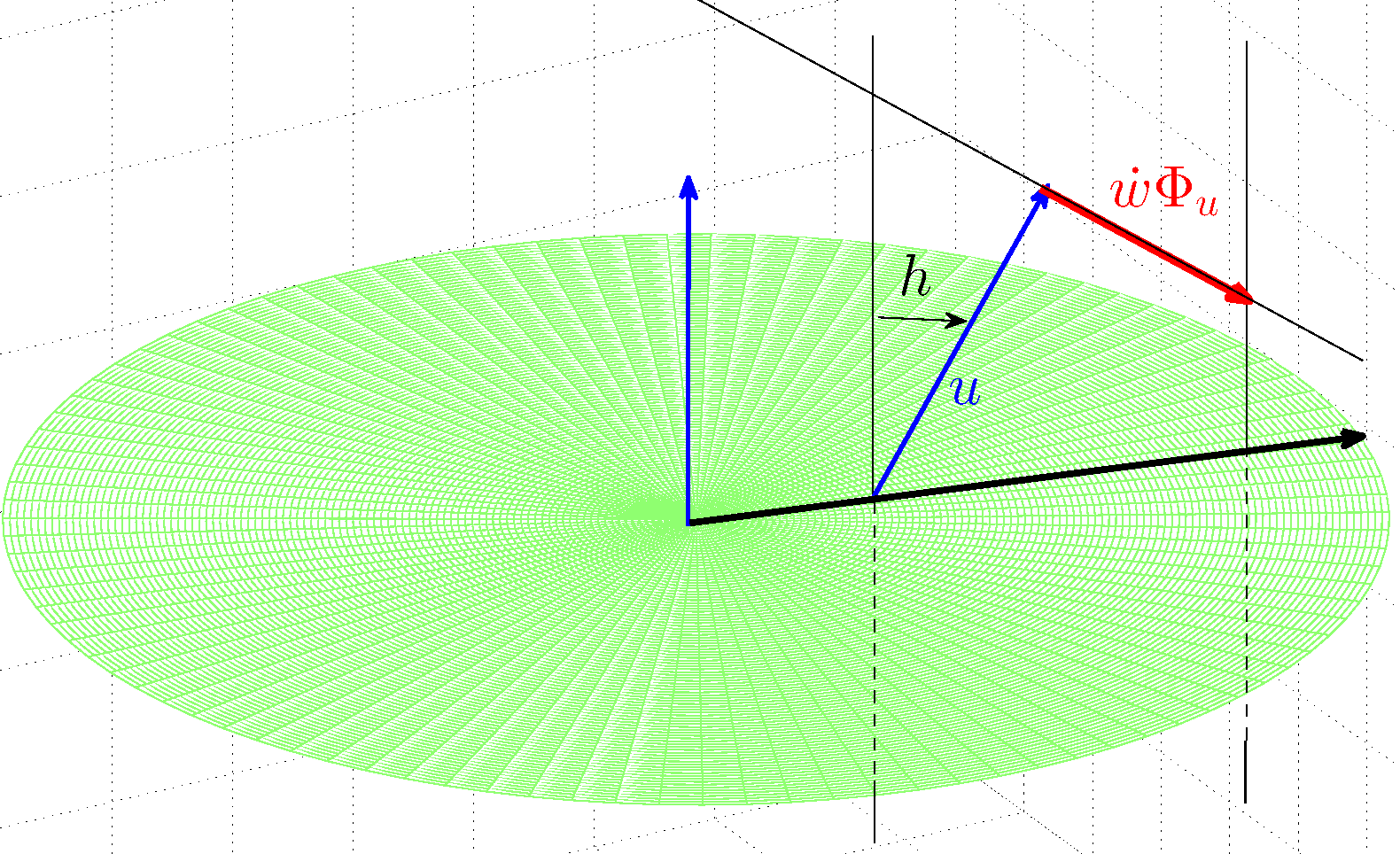}
\caption{``1-corotational noise'', represented in red.}
\label{fig:explication_equivariant}
\end{center}
\end{figure}
Fixing $T>0,$ this leads to the following $1$-corotational preserving version of \eqref{SHMF}:
\begin{equation}\label{SHMFprime}\tag{SHMF'}
\left\{
\begin{aligned}
&\d u=(\Delta u+u|\nabla u|^2)\d t+ \Phi_u\circ \d w_\phi  \enskip,\quad t\in(0,T]\times\D\,,
\\
& u|_{\{t\}\times \partial \D}=\mathbf k\,,\quad\quad\quad \quad\quad \quad    \text{for every}\enskip t\in[0,T]\,,
\\
& u(0)=u_{h_0}\,,
\end{aligned}\right.
\end{equation}
for a radially symmetric Wiener process $w_\phi:\Omega \times[0,T]\to L^2(\D;\mathbb R).$
More precisely, for $(t,x)\in[0,T]\times\D$ we let
\begin{equation}
\label{radially_sym_w}
w_\phi (t,x):=\sum_{k,j\geq 1} B_k(t)\phi_{k,j} e_j(|x|),
\end{equation}
for some coefficients $\phi _{k,j}\in\mathbb R,k,j\geq 1,$
where $(e_k)_{k\geq 1}$ denotes an orthonormal basis of the separable Hilbert space
\begin{equation}\label{nota:H}
H:=\left\{f:[0,1]\to\R:|f|_H^2:=\int_0^1f(r)^2r\d r<\infty\right\}\,,
\end{equation}
while the $B _k$'s are real-valued, independent and identically distributed Brownian motions on a given filtered probability space $(\Omega ,\F,\P,(\F_t)).$ Concerning the coefficients $(\phi _{k,j}),$ we shall at least assume in the sequel that the induced operator
$\phi :L^2(\D;\mathbb R)\to L^2(\D;\mathbb R),$ $h\equiv \sum_{k}h^ke_k(|\cdot |)\mapsto \sum_{j,k}h^k\phi _{k,j}e_j(|\cdot |)$ is  Hilbert-Schmidt.

To state our main results, we need to introduce a few notations.
Recall that the Dirichlet Laplacian $\Delta_D$ is the self-adjoint operator defined by the operation
$\partial _{11}+\partial _{22}$
on the domain
\[
D(\Delta_D ):= W^{2,2}(\D;\R)\cap W^{1,2}_0(\D;\R)\,,
\]
where $W^{2,2},W^{1,2}_0$ are the usual Sobolev spaces, the latter being composed only of elements that vanish on $\partial \D$, in the sense of traces.
We denote by $\left((-\Delta _D)^\alpha ,D\left((-\Delta _D)^{\alpha }\right)\right),\alpha \in\R$ its associated fractional powers defined through the Borel functional calculus,
and we let
\[
\mathscr H^{\alpha }:=D\left((-\Delta _D)^{\alpha /2}\right)\times D\left((-\Delta _D)^{\alpha /2}\right)\times W^{\alpha ,2}(\D;\R)
\]
Given $t\geq 0,$ we denote by $u(t)$ the trace of $u$ onto the time slice $\{t\}\times \D.$ 
For convenience in the statements below we define 
an extended state space $\mathscr {\bar H}^\alpha :=\mathscr H^\alpha \cup\{\vartriangle\},$ where $\vartriangle$ is an isolated point.
Finally, for a stopping time $\tau ,$ we denote by 
\[
\llbracket 0,\tau \rrparenthesis:=\left\{(\omega ,t)\in\Omega \times[0,T]:0\leq t<\tau (\omega )\right\}\,.
\]

We will agree on the following notion of solution.
\begin{definition}\label{def:weak_sol}
We will say that $(u,\tau )$ is a \emph{$1-$corotational, analytically weak solution} of \eqref{SHMFprime}
if $\tau >0$ is a stopping time and if the following properties hold.
\begin{enumerate}[label=(\roman*)]
 \item\label{def_i} The process $u:\Omega \times[0,T]\to \HH^1$ is progressively measurable and has continuous trajectories.
The singular time $\tau $ is caracterized by the property:
\[
u(\omega ,t)=\vartriangle\quad \text{if and only if}\enskip (\omega ,t)\notin \llbracket0,\tau \rrparenthesis\,.
\]
 Moreover, the solution $u$ fulfills the constraint
 \begin{equation}
 \label{constraint_def}
|u(\omega ,t,x)|=1\,,\quad\text{for}\enskip  \P\otimes \d t\otimes \d x-\text{a.e.}\enskip  (\omega ,t;x) \enskip \text{in}\enskip \llbracket0,\tau \rrparenthesis\times\D\,.
\end{equation}
 \item\label{def_ii} There exists a continuous, semi-martingale $h\equiv X+Y\in L^2(\Omega ;C(0,T;H))$ with respect to $(\F_t)$
where
\begin{itemize}
 \item the process $X$ has bounded variation in $V'$ where
$V'$ is the topological dual of the Hilbert space
\begin{equation}\label{nota:V}
V:=\Big\{\varphi \in H: |\varphi |_{V}^2:=\int_0^1 \big(\partial _r\varphi ^2+\frac{\varphi ^2}{r^2}\big)r\d r<\infty\Big\}\,;
\end{equation}
\item the quadratic variation process of $Y$ has finite trace in $H.$
\item for $\d \P\otimes \d t\otimes \d x-$a.e.\ $(\omega ,t;x)$ in $\llbracket0,\tau \rrparenthesis\times\D$ , there holds
\begin{equation}
\label{1-corotational-2}
u(t,x)=u_{h(t,|x|)}:=\left(\frac{x}{|x|}\sin h(t,|x|);\cos h(t,|x|)\right)\,.
\end{equation}
\end{itemize}
\item\label{def_iii} For every compactly supported $\psi$ in $C^1(\D;\R^3),$ and
almost every $(\omega ,t)\in\llbracket0,\tau\rrparenthesis,$ there holds
\begin{multline}\label{eq:weak_u}
\int_{\D}\Big(u(t,x)-u_0(x)\Big)\cdot \psi(x)\d x
=
-\iint_{[0,t]\times\D}\nabla u(s,x)\cdot \nabla \psi(x) \d x\d s
\\
+\iint_{[0,t]\times\D}u(s,x)\cdot \psi(x)|\nabla u(s,x)|^2\d x\d s
-\frac12\sum_{k\geq 1}\iint_{[0,t]\times\D}(\phi e_k(|x|))^2u(s,x)\cdot \psi(x)\d x\d s
\\
+\sum_{k\geq  1}\iint_{[0,t]\times\D}\psi (x)\cdot \Phi_{h(s,x)} \phi e_k(|x|)\d x\d B_k(s)\,,
\end{multline}
where $\Phi_h$ is the tangent vector defined in \eqref{frame}.
\item\label{def_iv} 
For every $t\geq 0,$
a.s.:
\[
u|_{\{t\}\times \partial \D}=\mathbf k\,,
\]
in the sense of traces.
\end{enumerate}
\end{definition}
With this definition at hand, we have the following.
\begin{theorem}[Existence, uniqueness and regularity of the solutions]\label{thm:loc_solv}
Let $T>0,$ fix $4>\beta>4/3,$
and consider a radially symmetric, $L^2(\D)$-valued Wiener process $w_\phi$ as in \eqref{radially_sym_w}.
Assume in addition that $\phi:L^2(\D)\to\mathscr H^{\beta'}$ is well-defined and Hilbert-Schmidt, for some $\beta '>\beta .$

For every $1-$corotational $u_0\equiv u_{h_0}$ in $\mathscr H^\beta$ (see \eqref{1-corotational-2}), there exist a stopping time $\tau^\beta(h_0)>0$ and a
\emph{$1-$corotational analytically weak solution} $(u,\tau ^\beta )$ of \eqref{SHMFprime}.
For this solution, we have the properties:
\begin{enumerate}[label=(\alph*)]
\item\label{sol_1} for $\P-$a.e.\ $\omega \in \Omega ,$ the path $u(\omega )$ belongs to the space $C([0,\tau^\beta);\mathscr H^\beta);$
\item\label{sol_2} if $\omega \in\Omega $ is such that $\tau (\omega )<T,$ then 
\[
\limsup_{t\nearrow\tau ^\beta (\omega )}|u(\omega ,t)|_{\H^\beta }=\infty\,.
\]
 \end{enumerate}
Moreover, the solution $(u,\tau ^\beta )$ is \emph{unique} in the class of $1-$corotational, analytically weak solutions fulfilling \ref{sol_1}-\ref{sol_2}.

Furthermore, the regularity propagates in the sense that
\[
\text{for every}\enskip \beta _1,\beta _2\in (2,4)\,,
\quad 
\tau^{\beta_1}=\tau^{\beta _2}\,.
\]
\end{theorem}
For the next theorem,
we will assume stronger assumptions on the operator $\phi $ (and the initial data).
This is due to the fact that, contrary to the previous statement, we want to ensure that the solution lives in $C^1(\D),$ from the very first time.
Note that, according to Remark \ref{rem:Sobolev_embedding}, a sufficient condition for this is that $u$ have continuous trajectories in $\H^{\beta }$ for some $\beta >2.$

In addition, it is crucial to add the non-degeneracy assumption
\eqref{non_degeneracy_assumption}
on $\phi ,$ so that the associated Ornstein-Uhlenbeck process has dense range (which will ensure irreducibility).
\begin{theorem}[Blow-up]\label{thm:main}
Let $\beta \in(2,4),$ and
let $T,\beta ',\phi ,w_\phi $ be as in Theorem \ref{thm:loc_solv}.
Assume moreover that:
\begin{equation}\label{non_degeneracy_assumption}
\ker\phi^*=\{0\}\,,
\end{equation}
and consider a $1-$corotational $u_0\equiv u_{h_0}$ in the space
$\mathscr H^\beta.$

Then, denoting by $(u,\tau^\beta) $ the solution built in Theorem \ref{thm:loc_solv} above (together with its maximal time of existence), for every $t_*>0,$ we have
\[\P\left(\tau^\beta \leq t_*\right)>0\,.\]
Additionally, blow-up happens in the following sense: for a.e.\ $\omega \in\Omega ,$
\begin{equation}
\label{blowup_happens}
\tau^\beta(\omega ) <T \quad\Longrightarrow\quad\limsup_{t\nearrow\tau^\beta(\omega )}|u(\omega ,t)|_{\mathscr H^{\beta _*}}=\infty\quad \text{for every}\enskip \beta _*>2\,.
\end{equation}
\end{theorem}
\begin{remark}
It is natural to expect that the solution constructed above actually lives in $\H^2,$ up to the singular time. It can be shown, according to a bootstrap argument (see \cite{hocquet2018struwe} for details in a slightly different setting), that provided $\sup_{t\in[0,\sigma )}|\Delta u(t)|_{L^2}<\infty$ for some stopping time $\sigma\in(0,T],$ then $u|_{[0,\sigma ]}$ has arbitrary regularity in space (with respect to what is allowed by the data $u_0,\phi $). Consequently, in \eqref{blowup_happens}, blow-up happens also for $\beta _*=2.$
\end{remark}


\paragraph{Outline of the paper.}
The proof of Theorem \ref{thm:loc_solv} will be adressed in Section \ref{sec:proof:solvability}.
Denoting the parabolic boundary by $\Sigma:=\{0\}\times[0,1]\cup[0,T]\times\{0,1\}$, then \eqref{SHMFprime} writes as an equation on the colatitude $h$ of $u$:
\begin{equation}\label{eq:h_dW_classical}
\d h=\Big(\partial_{rr}h+\frac{\partial_rh}{r}-\frac{h}{r^2}+\frac{2h-\sin2h}{2r^2}\Big)\d t+ \d w_\phi \enskip, \enskip\text{on}\enskip (0,T]\times(0,1)\,,
\end{equation}
where $(h-h_0)|_{\Sigma}=0.$
The fact that \eqref{eq:h_dW_classical} leads to \eqref{SHMFprime}, as well as the converse (namely that any $1-$corotational map $u\equiv u_h$ such that \eqref{SHMFprime} holds is such that $h$ verifies \eqref{eq:h_dW_classical}), will be justified by a generalized It\^o formula, a slightly different version of which can be found e.g.\ in \cite{debussche2016degenerate}.

Due to compensations, when $h$ is solution, we may have $\int_0^1(\frac{\partial _rh}{r}-\frac{\sin2 h}{2r^2})^2r\d r<\infty$ even if both terms of the integrand are not summable separately. This integral behaves as $\int_0^1(\frac{\partial _rh}{r}-\frac{h}{r^2})^2r\d r$, which motivates the introduction of the ``linearized Hamiltonian'' 
\[
A:=\partial _{rr}+\left(\frac{1}{r}\partial _r-\frac{1}{r^2}\right)\,,
\]
whose eigenpairs are related to the Bessel functions of the first kind.

The noise in \eqref{eq:h_dW_classical} is additive, and thus we have $h=v+z$, where $v$ solves the \emph{perturbed equation}:
\begin{equation}\label{perturbedEq}
\partial_t v=Av+\frac{2z+2v-\sin2(v+z)}{2r^2}\,,\enskip \text{on}\enskip (0,T]\times(0,1)\,,
\end{equation}
with $(v-h_0)|_{\Sigma}=0$,
and where $z=z(t,r)$ denotes a generic trajectory $Z(\omega )$ in the support of the solution
of the stochastic linear equation 
\[
\d Z=AZ\d t+\d w_\phi \,\enskip ,\enskip \enskip Z|_{\Sigma}=0\,.
\]
Theorem \ref{thm:loc_solv} will be proved using Picard's contraction mapping principle, at the level of the equation \eqref{perturbedEq} on a suitable subspace of $C([0,T];H)$ for $T$ small enough, using that the support of $Z$ as above is known.

Theorem \ref{thm:main} is proved in Section \ref{sec:proof_main_thm}.
Denoting by $h=h(h_0,Z)$ the local solution $v+Z$ of \eqref{eq:h_dW_classical}, it is obtained a consequence
of the existence of a ``nice'' pre-blow-up set $\mathfrak H$, namely a set of initial data $h_0$ such that:
(a) states in $\mathfrak H$ are reachable by the Markov Chain $h(h_0,Z,t)$ (in a sense precised below);
(b) the solutions starting from $h_0\in\mathfrak H$ blow up in finite time, with positive probability.
Part of the property (a) will be formulated in Section \ref{sec:proof:solvability}, as it appears as a natural consequence of the fixed point argument. Section \ref{sec:proof_main_thm} is mainly devoted to the proof of property (b) (whose precise statement is Proposition \ref{mainProp}). Proposition \ref{mainProp} is the core of the argument.
We point out that the topological argument used in Subsection \ref{subsec:globalization} to obtain existence of blowing-up solutions appears to be new in this context, and could perhaps be used for other SPDEs.
Technical facts related to local solvability and the comparison principle for \eqref{perturbedEq} are treated in the appendices.
\subsection{General notation and framework}
We denote by $\I$ the compact interval $[0,1]$.
For $1\leq p <\infty$, the notation $L^p_{r\d r}$ will be used to designate the Banach space of real valued measurable maps $r\mapsto f(r),\enskip r\in {\I}$, such that $|f|_{L^p_{r\d r}}:=(\int_I|f(r)|^pr\d r)^{1/p}<\infty$.
The special case $H=L^2_{r\d r}$, $|\cdot |_{H}$, defines a Hilbert space for the inner product
$f,g\in H\longmapsto\langle{f,g}\rangle =\int_If(r)g(r)r\d r$.

We need to introduce some functional spaces.
Let $A:D(A)\subset H\to H$ be the unbounded linear operator given by
\begin{empheq}[left=\empheqbiglbrack]{align}
\nonumber
&D(A)=\Bigg\{f\in H\,:\enskip \int_I\left[(\partial_{rr}f)^2+\big(\frac{\partial_rf}{r}-\frac{f}{r^2}\big)^2\right]r\d r<\infty\,,
\\
\label{nota:D_A}
&\quad \quad \quad \quad \quad \quad \quad \quad \quad 
\quad \quad \quad \quad \quad 
f\in C(\I) \enskip \text{and}
\enskip f(0)=f(1)=0\Bigg\},
\\
\label{nota:A}
&Ah=\partial_{rr}h+\left(\frac{1}{r}\partial_r-\dfrac{1}{r^2}\right)h\enskip,\quad h\in D(A)\,,
\end{empheq}
which has eigenpairs $\{(e_k,\lambda _k)\,,\,k\geq 1\}$ with $(e_k)$ forming an orthonormal basis of $H$, while the values $\lambda _k$ are negative and asymptotically quadratic in $k$ (see Section \ref{sec:proof:solvability}).
For $f,g\in D(A),$
we have the integration by parts formula:
\begin{equation}\label{ipp}
\int_{I}-Afgr\d r= \int_I\left(\partial_rf \cdot \partial_rg +\frac{f }{r}\frac{g }{r}\right)r\d r\equiv\int_I \nabla _r f\cdot \nabla _rgr\d r\,,
\end{equation}
where we introduce the nabla operator $\nabla _r:V\to H\times H,$ defined for $f\in V$ (see \eqref{nota:V}) as:
\begin{equation}\label{nota:gradient}
\nabla_r h(r):= \left(\partial _rh(r),\frac{h(r)}{r}\right)\,,\quad r\in\I\setminus\{0\}\,.
\end{equation}
As an immediate consequence, 
$A$ is symmetric. Furthermore, it is self-adjoint (see Section \ref{sec:proof:solvability}), so that from \cite[Theorem VIII.6]{reed1980methods}, we can define, when $\beta\geq 0$, the fractional power $(-A)^{\beta /2}.$ Explicitly, we have the formula
\begin{equation}\label{nota:A_beta}
(-A)^{\beta /2}h:=\sum_{k\in\mathbb{N}}(-\lambda _k)^{\beta/2 }\langle h\,,\,e_k\rangle e_k\,,
\end{equation}
for every $h$ in $V_\beta$ where
\begin{equation}\label{nota:V_beta}
V_\beta := D((-A)^{\beta /2})\equiv \left\{h\in H\,,\enskip  |h|_\beta ^2:=\sum\nolimits_{k\in\mathbb{N}}(-\lambda _k)^{\beta}\langle{h,e_k}\rangle ^2<\infty\right\}\,.
\end{equation}
At this point, it is worth noting, from \eqref{ipp}, that
\[
\langle-A f,f \rangle=|(-A)^{1/2}f|^2_{H}=|\nabla _r f|^2_{H}\,,
\]
and hence the topological spaces $V$ and $V_1\equiv D\left((-A)^{1/2}\right)$ coincide.

For $\beta \geq 0$, the norm in $C([0,T];V_{\beta})$ (i.e.\ the space of continuous functions with values in $V_\beta $), will be denoted by the double bars $\|\cdot \|_{T,\beta }$, namely if $z\in C([0,T];V_{\beta})$ we write
\begin{equation}\label{nota_double_bar}
\|z\|_{T,\beta}:=\sup_{0\leq t\leq T}|z(t)|_{\beta}\,.
\end{equation}

In the whole paper, we consider a filtered probability space $(\Omega,\mathcal{F},\P,(\mathcal{F}_t)_{t\geq0})$ satisfying the usual conditions. Note that the couples $(V_\beta ,|\cdot |_\beta )$ form separable Hilbert spaces, and thus by the classical theory of SPDE's \cite{DPZ},
the adapted $H$-valued Wiener process
\begin{equation}\label{wiener_process_chap4}
w_\phi(t)=\sum_{k\in\mathbb{N}}B_k(t)\phi e_k\enskip,
\end{equation}
where $(B_k)_{k\in\mathbb{N}}$ stands for a sequence of real-valued independent brownian motions in time,
$(e_k)_{k\in\mathbb{N}}$ is an ONB of $H$, and $\phi:H\to V_\beta $ is a Hilbert-Schmidt operator,
has continuous paths in the space $V_\beta$, with full probability.

The space of Hilbert-Schmidt operators from $H$ into some Hilbert space $K$ will be denoted by $\mathbb{L}_2(H;K)$.

Given a subset $S$ of a topological space $X$, we recall that the interior of $S$, which we denote by $\interior S$, consists of all points of S that do not belong to the boundary of $S$.

\begin{remark}\label{rem:equiv_A_Delta}
For $f\in H$, if $x\equiv r(\cos\theta,\sin\theta)\in\D$, define $F:\D\to\R^2$ by $F(x)=(f(r)\cos\theta,f(r)\sin\theta)$.
We have $|f|_{H}=(2\pi)^{1/2}|F|_{L^2(\D;\R^2)}$ and if $f\in V_2$, then $F\in D(\Delta )\equiv W^{2,2}\cap W^{1,2}_0$ with 
\begin{equation}\label{Delta_F}
\Delta F=(Af\cos\theta,Af\sin\theta)\,.
\end{equation}
Plugging the ansatz above in $\nabla ^2 F$, there holds in addition:
\begin{equation}\label{plugg_ansatz}
\int_{\D}|\nabla^2 F|^2\d x=2\pi\int_0^1\left(\partial_{rr}f\right)^2r\d r+4\pi\int_0^1\Big(\frac{\partial_rf}{r}-\frac{f}{r^2}\Big)^2r\d r\,.
\end{equation}
By a classical inequality, \eqref{plugg_ansatz} justifies that the norms
$|\partial_{rr}f|_H+\big|\left(\frac{\partial_r}{r}-\frac{1}{r^2}\right)f\big|_H$  and $|f|_2\equiv\enskip|Af|_H$, are in fact equivalent on $V_2$.

Furthermore, it follows from complex interpolation theory (see \cite[Chap.\ 4]{lunardi2009interpolation}) that
$F\in D\left((-\Delta )^{\beta /2}\right)^2$ if and only if $f\in V_\beta.$ 
Letting $f=\sin h$ for some $h\in V_2,$ then direct computations show that,
provided $h$ is bounded then
$|\nabla _rf|_{H}<\infty\Leftrightarrow |\nabla _r h|_{H}<\infty,$ hence $|(-A)^{1/2}f|_{H}<\infty\Leftrightarrow |(-A)^{1/2}h|_{H}<\infty.$ 
Therefore, by giving the characterization of Sobolev spaces in terms of complex interpolation, it holds that for $\beta >1:$
\begin{equation}
\label{equivalence_u_h}
u_h\in \H^\beta \quad \text{if and only if}\quad h\in V_\beta \,,
\end{equation}
(recall \eqref{1-corotational-2}).
\end{remark}
\begin{remark}\label{rem:Sobolev_embedding}
For $p\in[1,\infty),\beta\in\R$, $f\in V_\beta$, if $\beta<1$ and if
\[1\leq p\leq p^*=\frac{2}{1-\beta}\enskip,\]
the classical Sobolev Embedding Theorem in dimension $2$ (see \cite{adams2003sobolev}) implies that
$|F|_{L^p(\D)}\lesssim|(-\Delta )^{\beta/2}F|_{L^2(\D)}$, where we use the notations of Remark \ref{rem:equiv_A_Delta}. Since
$|F|_{L^p(\D)}=\left(2\pi\right)^{1/p}\left|f\right|_{L^p_{r\d r}}$, and $|(-\Delta )^{\beta/2}F|_{L^2(\D)}=|f|_{\beta}$, it is straightforward that we have the continuous embedding:
 $V_\beta\hookrightarrow L^p_{r\d r}$.
Similarly if $\beta>1$, then
$V_\beta\hookrightarrow C({\I};\R)$.
In addition, by the formula 
\begin{equation}\label{plugg_ansatz_nabla}
|\nabla F|^2=(\partial_rf)^2+\frac{f^2}{r^2}\,,
\end{equation}
then for any $\beta>2$, there exists $c_\beta>0$ such that for all $f\in V_\beta$,
$\left|\partial_rf\right|_{L^{\infty}}\leq c_\beta\left|f\right|_{\beta}$.
\end{remark}
\section{Proof of Theorem \ref{thm:loc_solv}}
\label{sec:proof:solvability}
\subsection{Treatment of the associated scalar problem, interpolation lemma}
Given $\beta\geq 0$ and $h_0\in V_\beta$
equation \eqref{eq:h_dW_classical} can be written as the infinite-dimensional equation in $H:$ 
\begin{equation}\label{eq:h_dW_functional}
\begin{cases}
\d h=\big(Ah+b(r,h(t,r))\big)\d t+ \d w_\phi\enskip ,& \text{for}\enskip t\in\R^+\enskip,
\\
h(0)=h_0\enskip,&
 \end{cases}
\end{equation}
where `` $\d $ '' denotes It\^{o} differential, whereas for $f\in H$, we denote by $b(r,f(r))$ the nonlinearity:
\begin{equation}\label{nota:b}
b(r,f(r))=\dfrac{2f(r)-\sin2f(r)}{2r^2}\,,\enskip r\in \I\setminus\{0\}\,,
\end{equation}
which will be sometimes abbreviated as $b_f.$
Note that $b_f$ is not always well-defined as an element of $H.$
However, assuming the existence of $\beta >4/3,$ such that $f\in V_\beta ,$ then
$|b_f|_{H}<\infty$
(see \eqref{P2} below) hence $b_f$ is well defined.

It is well known that the spectrum of $A$ (identified with its standard complexification) consists exclusively of eigenvalues whose associated vectors are proportionnal to $J_1(x_k\cdot ),k\geq 1,$ where $J_1$ is the \emph{order one Bessel function of the first kind}, smooth solution to the ODE
\[
\begin{cases}
  y^2\dfrac{\d ^2J_1}{\d y^2}+y\dfrac{\d J_1}{\d y}+(y^2-1)J_1=0\enskip,\quad \text{for}\enskip y>0\enskip,
  \\
  J_1(0)=0\,,
 \end{cases}
\]
and where $(x_k),k\geq 1$ is a countable part of $\R_+\setminus0,$ forming the zeros of $J_1.$ It is a well known fact that, if we arrange them in ascending order (we will do this assumption in the sequel), then the $x_k$'s are asymptotically linear in $k\in\mathbb{N}^*$.
For $k\in\mathbb{N}^*$, the mappings
\begin{equation}
\label{nota:ek}
e_k:=\left(r\mapsto \frac{1}{|J_1(x_k\cdot)|_H}J_1(x_kr)\enskip,\enskip r\in{\I}\right)\enskip,
\end{equation}
define a family $(e_k)_{k\in\mathbb{N}^*}$ of eigenvectors of $A$, with associated eigenvalues $\lambda _k:=-(x_k)^2$, $k\in\mathbb{N}^*$, which forms an orthonormal basis of $H.$

In particular, the spectrum of $A$ lies in a sector containing the negative real axis, so that it generates an analytic semigroup $S.$ Furthermore, the following inequality holds, for any $\beta \geq 0:$
\begin{equation}
\label{P1}
|(-A)^{\alpha } S(t)|_{\mathscr{L}(V_{\beta})}\leq c t^{-\alpha }\,,\quad \text{for all}\enskip t>0\enskip\text{and every}\enskip \alpha\geq 0\,,\\
\end{equation}
for a constant depending on $\alpha$ (we refer the reader, e.g.\ to \cite[\S 2.6]{pazy1983semigroups}).

The proof of local solvability relies mainly on \eqref{P1}, together with
suitable estimates of $b_h$ in the scale $(V_\beta )$ (see Corollary \ref{cor:P} below).
The following interpolation lemma, which is needed for obtaining such estimates,
is based on expansion of functions in terms of their so-called Fourier-Bessel series (see \cite[chap.~18]{watson1995treatise}), that is according to the basis defined in \eqref{nota:ek}.
\begin{lemma}\label{lem:alpha}
Let $p\in[1,\infty]$, $\nu\in\R$.
\begin{enumerate}[label=(\roman*)]
 \item\label{lem:alpha:item_i}
Take $\nu\leq2/p+1$ and define the operator $T:D(T)\subset V_\beta\rightarrow L^p_{r\d r}$ by $Tf=\dfrac{f}{r^\nu}$ for 
$f\in D(T):=\{\varphi\in V_\beta, \left|\frac{\varphi}{r^\nu}\right|_{L^p_{r\d r}}<\infty\}$.

Provided $\beta>(1+\nu-2/p)\vee1/2$, then $T$ has a bounded extension
\[
\begin{cases}
T:V_\beta \longrightarrow L^p_{r\d r}\quad \text{if}\enskip p<\infty\enskip \text{and}\enskip \nu<2/p+1\,;\\
T:V_\beta \longrightarrow L^{\infty}\quad \text{if}\enskip p=\infty\enskip \text{and}\enskip \nu\leq 1\,.
\end{cases}
\]
\item\label{lem:alpha:item_ii}
 Similarly, the linear map
$\partial_r:D(\partial_r):=\{\varphi\in V_\beta,|\partial_r\varphi|_{L^p_{r\d r}}<\infty\}\longrightarrow L^p_{r\d r}$,
has a bounded extension $\partial_r: V_\beta\rightarrow L^p_{r\d r},$
provided $\beta>(2-2/p)\vee 3/2.$
\end{enumerate}
\end{lemma}
\begin{proof}[Proof of Lemma \ref{lem:alpha}]
\indent\textit{Case $p=\infty$}
According to \eqref{plugg_ansatz_nabla} the following bound holds, provided $\beta >2$:
\[
\sup_{r\in \I}\big\{(\partial _rf(r))^2 +\frac{f(r)^2}{r^2}\big\}\leq c_\beta |f|_{\beta }^2\,,\quad\text{for all}\enskip f\in V_{\beta }\,,
\]
for some $c_\beta >0$. This yields both \ref{lem:alpha:item_i} (when $\nu \leq 1$) and \ref{lem:alpha:item_ii}.

\item[\indent\textit{Proof of \ref{lem:alpha:item_i}}.]
Let $p\in[1,\infty)$.
Using the orthonormal basis defined in \eqref{nota:ek}, and for $k\geq1$ setting $c_k:=|J_1(x_k\cdot )|_H^{-1}$, one has by \eqref{nota:ek}:
\[\big|\frac{1}{r^\nu} e_k\big|_{L^p_{r\d r}}^2=(c_k)^2\big|\frac{1}{r^\nu}J_1(x_k\cdot )\big|_{L^p_{r\d r}}^2
=(c_k)^2(x_k)^{2\nu-4/p}\left(\int_0^{x_k}\frac{|J_1(y)|^p}{y^{p\nu-1}}\d y\right)^{2/p},\]
where we have done the change of variable $y=x_kr$.
 Using classical properties of Bessel functions, see \cite[chap.~7]{watson1995treatise}, there exist constants $c,c'>0$ such that
 \begin{equation}\label{bessel_prop}
 J_1(y)\leq cy\,,\enskip\forall y\in{\I}\quad \text{and}\quad|J_1(y)|\leq c'y^{-1/2}\,, y\in[1,\infty)\,.
\end{equation}
By \eqref{bessel_prop}, we obtain that for $\nu\in\R$, provided every term below is finite:
\begin{equation}\label{ineq:asymptotics_int_J1}
\int_0^{x_k}\frac{|J_1(y)|^p}{y^{p\nu}}y\d y\leq c''\left(\int_0^1y^{p-p\nu+1}\d y+\int_1^{x_k}y^{-p/2-p\nu+1}\d y\right)\,.
\end{equation}
Since $x_k$ is asymptotically linear in $k\geq 1$ (\cite[p.\ 503-510]{watson1995treatise}), the right hand side of \eqref{ineq:asymptotics_int_J1} remains bounded independently of $k$ if and only if $2/p-1/2<\nu<2/p+1.$
Noticing furthermore that $(c_k)^2\equiv|J_1(x_k\cdot )|_{H}^{-2}=\mathcal{O}(k)$ (this is left to the reader), we have
\[\big|\frac{1}{r^\nu}e_k\big|_{L^p_{r\d r}}^2=
\begin{cases}
\mathcal{O}(k^{1+2\nu-4/p})\quad\text{if}\enskip \frac2p-\frac12<\nu<\frac{2}{p}+1\,,\\
\mathcal{O}(1)\quad\text{if}\enskip \nu\leq\frac2p-\frac12\,.
\end{cases}
\]
Using now triangle and Cauchy-Schwarz inequalities on the Fourier-Bessel series of $f\in V_\beta$, we have formally
\begin{equation}\label{ineq:T_L2_Lp}
|Tf|_{L^p_{r\d r}}\leq\sum_{k\geq1}|\langle{f,e_k}\rangle ||Te_k|_{L^p_{r\d r}}
\leq|f|_{\beta}\Big(\sum_{k\geq1}(x_k)^{-2\beta}|Te_k|_{L^p_{r\d r}}^2\Big)^{1/2}.
\end{equation}
Taking $\beta>1+\nu-2/p$ gives a convergent series in \eqref{ineq:T_L2_Lp} in the case $2/p-1/2<\nu<2/p+1$, whereas $\beta >1/2$ is sufficient when $\nu\leq2/p-1/2$.
In both cases, we obtain a continuous extension $T:V_{\beta}\rightarrow L^p_{r\d r}.$

\item[\indent\textit{Proof of \ref{lem:alpha:item_ii}}.]
The proof of the second assertion is similar,
using that $|\partial_re_k|_{L^p_{r\d r}}=c_kx_k|J_1'(x_k\cdot )|_{L^p_{r\d r}}$.
The well-known identity 
$J_1'(y)=J_0(y)-J_1(y)/y,$  $y\geq0$ (see \cite[p.~17-19]{watson1995treatise}),
shows in particular that $J_1'(x_k\cdot )$ defines an element of $L^p_{r\d r}$ near the origin. For some constant $c>0$ we obtain
\begin{equation}\label{int_J1prime}
\left|\partial_re_k\right|_{L^p_{r\d r}}^2\leq ck^{3-4/p}\left(|J'_1|_{L^p_{r\d r}([0,1])}^p+\int_1^{x_k}\Big|J_0(y)-\frac{J_1(y)}{y}\Big|^py\d y\right)^{2/p}\,.
\end{equation}
Now, as for $J_1$ there exists $c'>0$, such that:
$J_0\leq c'y^{-1/2}$, the other term $J_1/y$ being smaller at infinity.
Therefore, in case $p>4$, the integral in \eqref{int_J1prime} is bounded, so that inequality \eqref{ineq:T_L2_Lp}, with $T:=\partial_r$ and $\beta>2-2/p$, gives the result.
Otherwise if $p\in[1,4]$, we have $|\partial _re_k|^2_{L^p_{r\d r}}=\mathcal O(k^{4/p-1})$, and it is sufficient to take $\beta >3/2$.
\end{proof}
We can now state our main estimates on $b.$
\begin{corollary}
\label{cor:P}
For $\beta >4/3,$
we have the estimates
\begin{align}
\label{P2}
&|b_{v}|_{H}\leq c'|v|_{\beta}^3\,,\qquad\text{for all}\enskip v\in V_\beta \,,
\\
\label{P3}
&|b_{u}-b_{v}|_{H}\leq c''|u-v|_{\beta}(|u|_{\beta}^2+|v|_{\beta}^2)\,,\quad\text{for all}\enskip u,v\in V_{\beta} \,,
\end{align}
with constants depending on $\beta$ only.
\end{corollary}
\begin{proof}[Proof of Corollary \ref{cor:P}]
Denoting by
$F:x\mapsto x-\sin2x/2,\enskip x\in \R$, and using the inequality $|F(x)|\leq c|x|^3$, $x\in\R$, for a certain $c>0$, we have by an application of Lemma \ref{lem:alpha}--\ref{lem:alpha:item_i} with $\nu=2/3$, $p=6$:
\[\big|b_{v}\big|_{H}=\big|\frac{F(v)}{r^2}\big|_{H}\leq c\big|\frac{v}{r^{2/3}}\big|_{L^6_{r\d r}}^3\leq c'\big|v\big|_{\beta}^3\enskip,\]
as soon as $\beta>4/3$, which shows \eqref{P2}.

Similarly, using that for some $c>0$,
$|F(x)-F(y)|\leq c|x-y|(x^2+y^2)$, $\forall x,y\in\R$, then H\"{o}lder's inequality implies:
\[
\begin{aligned}
|b_{u}-b_{v}|_{H}&\leq c\Big|\big|\frac{u-v}{r^{2/3}}\big|\Big(\big(\frac{u}{r^{2/3}}\big)^2+\big(\frac{v}{r^{2/3}}\big)^2\Big)\Big|_{H}\\
&\leq c\big|\frac{u-v}{r^{2/3}}\big|_{L^6_{r\d r}}\Big(\big|\frac{u}{r^{2/3}}\big|_{L^6_{r\d r}}^2+\big|\frac{v}{r^{2/3}}\big|_{L^6_{r\d r}}^2\Big)\,.
\end{aligned}
\]
An application of Lemma \ref{lem:alpha}--\ref{lem:alpha:item_i} with the same parameters as above leads to \eqref{P3}.
\end{proof}

\subsection{Construction of a mild solution}
In view of the previous analysis,
we will first solve \eqref{eq:h_dW_functional} under the \emph{mild form:}
\begin{equation}\label{eq:mild_h}
h(t)=S(t)h_0 + \int_0^tS(t-s)b_{h}(s)\d s + \int_0^t S(t-s)\d w_\phi(s)\,,\enskip\text{for}\enskip t\in[0,\tau^\beta)\,,\enskip \text{a.s.}\,,
\end{equation}
where $S\equiv e^{\cdot A}$ is the semigroup generated by $A$, each integral above being understood in the Bochner sense, in $V_\beta $ for some $\beta >4/3.$
More precisely, our aim here is to show the following.
\begin{claim}\label{clm:LocalSolvability-h}
Let $4>\beta>4/3$, $\beta '>\beta ,$
and take $\phi\in\mathbb{L}_2(H,V_{\beta'}).$
Then, for $h_0\in V_\beta$, there exist a stopping time $\tau^\beta(h_0)$, and a unique $h$ with paths in $C([0,\tau^\beta);V_\beta)$, a.s.~, mild solution of \eqref{eq:h_dW_functional}.
The stopping time $\tau^\beta$ is maximal in the sense that for every $R>0,$ $\tau ^\beta \geq  \inf\{t\in[0,T], |h(t)|_{\beta }\geq R\}.$
\end{claim}
\begin{proof}[Proof of Claim \ref{clm:LocalSolvability-h}]
We restrict our proof to the case $\beta \in(4/3,2]$. Higher regularity, as well as the propagation, are treated in Appendix \ref{app2}.
Fix $T>0$. For $\omega\in\Omega$ and $t\in[0,T]$, define the Ornstein-Uhlenbeck process
\begin{equation}\label{nota:Z}
Z(\omega ,t)=\int_0^tS(t-s) \d w_\phi(s)\,,\enskip\omega\in\Omega\,,\enskip t\geq 0\,.
\end{equation}
For an analytical semi-group $S,$ since by assumption $\phi\in\mathbb{L}_2(H,V_{\beta'})$ with $\beta '>\beta -1,$ then
it is standard that:
\[
Z\enskip  \text{is a random variable supported in the space}\quad C([0,T];V_{\beta})\,,
\]
see \cite[\S6]{DPZ}.
Therefore we can take $z\in C([0,T];V_{\beta})$, and argue pathwise, 
considering the translated equation \eqref{perturbedEq} with unknown $v$.
For $h_0\in V_{\beta}$, if a solution $v$ exists up to $\tau =\tau (z)>0$, it is well-known that $h:=v+z$ gives a solution of \eqref{eq:h_dW_functional} on $\{Z|_{[0,\tau ]}=z|_{[0,\tau ]}\}$.
Thus, for each $z\in C([0,T];V_{\beta})$, we aim to find a fixed point $v$ for the map $\Gamma=\Gamma_{h_0,z,T}$, defined as
\begin{equation}\label{nota:contraction}
\Gamma(v)(t):=S(t)h_0+\int_0^tS(t-s)b_{v+z}(s)\d s\,,\enskip \text{for}\enskip t\in[0,T]\,.
\end{equation}
We will show that if $T_*>0$ is sufficiently small, depending only on $\|z\|_{T,\beta}$ and $|h_0|_{\beta}$, then the mapping $\Gamma$ is a contraction of a certain ball of $C([0,T_*];V_{\beta})$.

Consider any $z$ as above, and $h_0\in V_\beta$.
If $v\in C([0,T];V_{\beta})$, taking the $V_\beta$-norm in \eqref{nota:contraction} and using \eqref{P1} and \eqref{P2} gives:
\begin{equation}
\label{nota:T_star_1}
 \|\Gamma(v)\|_{T,\beta}\leq |h_0|_\beta+c_1T^{1-\beta/2}(\|v\|_{T,\beta}^3+\|z\|_{T,\beta}^3)\,.
\end{equation}
Then, using \eqref{nota:contraction}, for $u,v\in C([0,T];V_{\beta})$, we have by \eqref{P1} and \eqref{P3}:
\begin{equation}\label{nota:T_star_2}
\|\Gamma(u)-\Gamma(v)\|_{T,\beta}\leq c_2T^{1-\beta/2}(\|u\|^2_{T,\beta }+\|v\|_{T,\beta}^2+2\|z\|_{T,\beta}^2)\|u-v\|_{T,\beta}\,.
\end{equation}
Set $R:=|h_0|_\beta\vee \|z\|_{T,\beta}+1.$
Letting
\begin{equation}\label{T_star}
T_*:=\min\left(\frac{1}{4c_1R^3},\frac{1}{8c_2R^2}\right)^{1/(1-\beta /2)}\,,
\end{equation}
then \eqref{nota:T_star_1} and \eqref{nota:T_star_2} yield respectively $\|\Gamma(v)\|_{T_*,\beta }\leq R-1/2$ and $\|\Gamma(u)-\Gamma(v)\|_{T_*,\beta }\leq(1/2)\|u-v\|_{T_*,\beta }$, so that:
\begin{itemize}
 \item the ball $\mathbb B^{R}\subset C([0,T_*];V_\beta )$ centered at $0$ and of radius $R$ is left invariant by $\Gamma_{h_0,z,T_*}$;
 \item $\Gamma_{h_0,z,T_*}:\mathbb B^{R}\to\mathbb B^{R}$ is a contraction.
\end{itemize}
Applying Picard Theorem (the underlying space is complete), there exists a unique fixed point $v(h_0,z)$ for $\Gamma_{h_0,z,T_*}$, a mild solution to the perturbed equation \eqref{perturbedEq}, up to $t=T_*$. The maximal solution is obtained by reiteration of the argument.
\end{proof}
Fixing $h_0\in V_\beta$ and  $z\in C([0,T_*];V_\beta )$,
the above proof shows that if $R:=|h_0|_\beta\vee \|z\|_{T,\beta}+1$ and $T_*(R)$ is as in \eqref{T_star}, then the unique fixed point of $\Gamma_{h_0,z,T_*}|_{\mathbb B^R}$, which we denote by $v_0$, depends continuously on $z|_{[0,T_*]}$ and $h_0$.
Indeed, first note that by \eqref{T_star} we have
\begin{equation}\label{relation_fix}
\|\Gamma_{h_1,\zeta}(v)\|_{T_*,\beta }\leq R-\frac14\enskip \text{and}\enskip \|\Gamma_{h_1,\zeta }(u)-\Gamma_{h_1,\zeta }(v)\|_{T_*,\beta }\leq\frac34\|u-v\|_{T_*,\beta },
\end{equation}
for $(h_1,\zeta )$ lying in some neighbourhood $\mathcal{V}\times\mathcal{W}$ of $(h_0,z)$. By the previous analysis, the bound \eqref{relation_fix} guarantees the existence of the unique fixed point $v_1$ of $\Gamma _{h_1,\zeta ,T_*}$.
For such $(h_1,\zeta )$, using that  $v_0=\Gamma_{h_0,z}(v_0)$, $v_1=\Gamma_{h_1,\zeta }(v_1)$, and re-using the properties \eqref{P1}-\eqref{P2}-\eqref{P3}, we immediately obtain
\[\|v_0-v_1\|_{T_*,\beta}\leq |h_0-h_1|_\beta
 +cT_*^{1-\beta/2}R^2
\Big(\|v_0
 -v_1\|_{T_*,\beta}+\|z-\zeta\|_{T_*,\beta}\Big)\enskip,\]
so that the continuity of $v$ at $(h_0,z)\in \mathcal{V}\times\mathcal{W}$ follows. This eventually gives the continuity for $h:=v+z$, locally on $[0,T_*]$.
The continuity of these functionals remains true up to the maximal times, as stated in the next lemma (the proof is postponed to Appendix \ref{app2}).
\begin{lemma}[Continuous dependence]\label{lem:c_dep_result}
Let $T>0$, $z\in C([0,T];V_{\beta})$, $h_0\in V_\beta$ and assume that $h(h_0, z,\cdot )$ exists on $[0,T]$.
There exist open sets $\mathcal{V}\subset V_\beta $ and $\mathcal{W}\subset C([0,T];V_{\beta})$, with $(h_0,z)\in\mathcal{V}\times\mathcal{W}$, such that
for all $(h_1,\zeta )\in\mathcal{V}\times\mathcal{W},$ there exists a unique mild solution $h(h_1,\zeta ,\cdot )\in C([0,T];V_{\beta})$ of \eqref{eq:h_dW_functional}.

Moreover, the mapping 
$\mathcal{V}\times\mathcal{W}\to C([0,T];V_{\beta})$, $(h_1,\zeta) \mapsto h(h_1,\zeta ,\cdot )|_{[0,T]}$,
is continuous.
\end{lemma}
\subsection{Equivalence of formulations and conclusion}
To show equivalence between the formulations \eqref{eq:mild_h} and \eqref{eq:weak_u}, it is convenient to revert to an interpretation of a solution with a more ``variational appeal'' (see e.g.\ the definition of a solution in \cite{prevot2007concise}).
For $\beta >4/3,$ 
if $(h,\tau )$ is the mild solution built in Claim \ref{clm:LocalSolvability-h},
using \cite[Proposition 6.4]{DPZ}, it is true that $(h,\tau )$ is also a
\emph{weak solution}, namely: 
for every $\zeta $ in $D(A),$ we have
\begin{equation}\label{weak_solution3}
\langle h(t)-h_0,\zeta \rangle=
\int_0^t\langle h,A\zeta \rangle\d s+
\int_0^t\langle b_{h},\zeta \rangle\d s
+\int_0^t\langle \d w_\phi ,\zeta\rangle\,,
\qquad \text{a.s.\ for} \enskip t<\tau\,.
\end{equation}
Furthermore,
we can extend \eqref{weak_solution3} to test functions that belong to the larger space $V\equiv D((-A)^{1/2}).$ Indeed, 
from the inequality $|\sin2h|\leq 2|h|,$ since
\[
\int_I((\partial _r\zeta )^2+\frac{\zeta ^2}{r^2})r\d r<\infty,\quad \forall \zeta \in V\,,
\]
then one can integrate by parts (see \eqref{ipp}) in the first term of \eqref{weak_solution3}, the resulting expression depending continuously of $\zeta $ in $V.$ Using a density argument, we have for every $\zeta$ in $V:$
\begin{equation}\label{weak_solution}
\langle h(t)-h_0,\zeta \rangle=-\int_0^t\langle \partial _{r}h,\partial _r\zeta \rangle\d s -\int_0^t\left\langle \frac{\sin2h}{r} ,\frac{\zeta }{r}\right\rangle\d s
+\int_0^t\langle \d w_\phi ,\zeta  \rangle\,,\qquad \text{a.s.\ for} \enskip t<\tau\,,
\end{equation}
which extends \eqref{weak_solution3}.

Now, for a solution of \eqref{weak_solution}, the following generalized It\^o Formula holds true:
let $\zeta  \in C^1((0,1);\R),$ compactly supported, and $\varphi\in C^2(\R^3)$ with bounded derivatives, then we have:
\begin{multline}\label{ito_formula}
\big\langle\varphi (h(t))-\varphi (h_0),\zeta  \big\rangle=
-\int_0^t \big\langle \nabla _rh(s),\nabla _r\left(\varphi' (h(s))\zeta \right)\big\rangle \d s
+\int_0^t \big\langle  b_{h}(s),\varphi' (h(s))\zeta \big\rangle \d s
\\
+\sum_{k\geq 1}\int_0^t \big\langle\varphi '(h(s))\phi e_k, \zeta  \big\rangle\d B_k 
+\frac12\sum_{k\geq 1}\int_{0}^t \big\langle\varphi ''(h(s))(\phi e_k)^2,\zeta  \big\rangle\d s\,,
\end{multline}
a.s.\ for $t\in[0,T].$
Note that every term above makes sense, since on the one hand:
\[
|\varphi '(h)\zeta |^2_{V}
\equiv \int_I\left[\big(\varphi ''(h)\partial _rh\zeta  +\varphi'(h)\partial _r\zeta  \big)^2+\left(\frac{\varphi '(h)\zeta }{r}\right)^2\right]r\d r <\infty\,,
\]
and, as already noticed in \eqref{P2}, we have on the other hand 
$|b_{h}|_{H}\leq c|h|_{V_\beta }^3<\infty.$
\begin{proof}[Proof of \eqref{ito_formula}]
The proof follows the lines of \cite[Prop.~A.1.]{debussche2016degenerate}, the difference being that: (i) functional spaces here consist of radial functions;
(ii) the mollification argument needs to fit with the data on the boundary (hence we need to suitably extend $h$ on a bigger space interval).

Define the ``1-fattening'' of a given $f\in V$ as the map $\tilde f:[-1,2]\to \R$ such that
\begin{equation}\label{1-fattening}
\tilde f(r):=
\begin{cases}
f(r)\quad\text{if}\quad r\in[0,1]
\\
-f(-r)\quad\text{if}\quad r\in[-1,0)
\\
-f(2-r)\quad\text{if}\quad r\in(1,2]\,.
\end{cases}
\end{equation}
For $\zeta \in C^\infty_0(-1,1)$ \eqref{weak_solution},
consider a decomposition 
\begin{equation}\label{decomposition_zeta}
\zeta = \zeta_-+\vartheta +\zeta _+,\quad \text{where:}
\quad\spt\zeta_-\subset (-1,0),
\quad\vartheta \enskip \text{is even}\enskip
\quad\text{and}\enskip \spt\zeta _+\subset(0,1)\,,
\end{equation}
and denote by $\zeta_0:=\zeta_-+\zeta _+.$
Owing to \eqref{weak_solution} and \eqref{1-fattening}, we have
\begin{multline*}
\int_{-1}^0\left(\tilde h(t,r)-\tilde h_0(r)\right)\zeta _-(r)r\d r=\int_0^1\left(h(t,r)-h_0(r)\right)\zeta _-(-r)r\d r
\\
\equiv
-\int_0^t\big\langle \partial _{r}h,\partial _r\left(\check\zeta_-\right) \big\rangle\d s -\int_0^t\left\langle \frac{\sin2h}{r} ,\frac{\check\zeta_- }{r}\right\rangle\d s
+\int_0^t\left\langle \d w_\phi ,\check\zeta_- \right\rangle\,,
\end{multline*}
where $\check\zeta _-:=\zeta _-(-\cdot ),$
but on the other hand, we have for instance
\begin{multline*}
-\int_0^t\left\langle \partial _rh,\partial _r(\check\zeta _-)\right\rangle \d s
=+\iint_{[0,t]\times[0,1]}\partial _rh(r)\partial _r\zeta _-(-r)r\d r\d s
\\
=-\iint_{[0,t]\times[-1,0]}\partial _r\tilde h(r)\partial _r\zeta _-(r)r\d r\d s\,.
\end{multline*}
The treatment of the other terms is similar, hence
summing over $\zeta _+,\zeta _-,$ we obtain
\begin{multline}\label{weak_solution2}
\int_{[-1,2]}(\tilde h(t)-\tilde h_0)\zeta_0 r\d r 
=-\iint_{[0,t]\times [-1,2]}\partial _{r}\tilde h\partial _r\zeta_0 r\d r\d s 
\\
-\iint_{[0,t]\times [-1,2]}\frac{\sin2\tilde h}{r} \frac{\zeta_0}{r}r\d r\d s
+\sum_{k\geq 1}\iint_{[0,t]\times [-1,2]}\widetilde{(\phi e_k)}\zeta_0 r\d r\d B_k \,,
\end{multline}
a.s.\ for $t<\tau.$

Furthermore,
using the fact that $\vartheta $ is even, it holds true that
\begin{align}
\label{id_theta}
\int_{[-1,0]}(\tilde h(t)-\tilde h_0)\vartheta r\d r =-\int_{[0,1]}(\tilde h(t)-\tilde h_0)\vartheta r\d r 
\\
-\iint_{[0,t]\times [-1,0]}\partial _{r}\tilde h\partial _r\vartheta r\d r\d s 
=\iint_{[0,t]\times [0,1]}\partial _{r}\tilde h\partial _r\vartheta r\d r\d s 
\\
\sum_{k\geq 1}\iint_{[0,t]\times [-1,0]}\widetilde{(\phi e_k)}\vartheta r\d r\d B_k
=-\sum_{k\geq 1}\iint_{[0,t]\times [0,1]}\widetilde{(\phi e_k)}\vartheta r\d r\d B_k \,,
\intertext{
and concerning the singular term,
we have for every $0<\epsilon <1$:}
\label{id_theta2}
-\iint_{[0,t]\times [-1,-\epsilon ]}\frac{\sin2\tilde h}{2r} \frac{\vartheta}{r}r\d r\d s
=\iint_{[0,t]\times [\epsilon ,1]}\frac{\sin2\tilde h}{2r} \frac{\vartheta}{r}r\d r\d s\,,
\end{align}
which suggest that one can simply \emph{define}, for every $\zeta \in C^\infty_0(-1,1)$ the integral
\begin{equation}
\label{cauchy_principal}
-\iint_{[0,t]\times [-1,1]}\frac{\sin2\tilde h}{2r} \frac{\zeta }{r}r\d r\d s
:=\lim_{\epsilon \to0}-\iint_{[0,t]\times [-1,-\epsilon ]\cup[\epsilon ,1]}\frac{\sin2\tilde h}{2r} \frac{\zeta }{r}r\d r\d s
\end{equation}
(note the ressemblance with the notion of ``Cauchy principal value'').
Similar relations hold if $\vartheta $ is supported around $r=1,$ changing the intervals from $[-1,0],[0,1]$ to $[0,1],[1,2].$
Hence, \eqref{weak_solution2} can be extended to test functions that are not necessarily equal to zero at the origin.

Now, in \eqref{weak_solution2}, for every $r\in \I$, test against $\rho _\epsilon (r- \cdot ),$ where $(\rho _\epsilon )$ is an approximation of the identity such that for every $\epsilon >0:$
\[
\rho _\epsilon \enskip \text{is even}\quad\text{and}
\quad \spt\rho _\epsilon \subset (-1,1)\,.
\]
For $f$ in $V,$ denote by $f^\epsilon (r)$ the convolution $\int_{[-1,2]} \tilde f(r')\rho _\epsilon (r-r')r'\d r',$ for $f$ in $H,$ and observe,
using self-adjointness, that
\[
A(f^\epsilon (r))\equiv\int_{[-1,2]} \tilde f(r')A\rho _\epsilon (r-r')r'\d r'
=\int_{[-1,2]} A\tilde f(r')\rho _\epsilon (r-r')r'\d r'
=(Af)^\epsilon(t,r) \,.
\]
We hence obtain that
\begin{equation}
h^\epsilon (t,r)=h_0^\epsilon (r)+\int_0^tA(h^\epsilon (s,r))\d s
+\int_0^t \Big(\frac{2h -\sin 2h}{2r^2}\Big)^\epsilon(s,r)\d s
+\sum_{k\in \N}\int_0^t\phi _k^\epsilon (r)\d B_k(s)
\end{equation}
where we define $\phi _k:=\phi e_k.$
Applying the 1-dimensional It\^o Formula, multiplying by $r\zeta (r)$, $\zeta \in C^1_0(\I)$ and integrating over $r\in \I,$ we end up with
\begin{multline}
\label{1D-Ito}
\left\langle \varphi\big(h^\epsilon (t)\big)-\varphi \big(h_0^\epsilon \big),\zeta \right\rangle
\\
=-\int_0^t\big\langle \partial _r h^\epsilon ,\partial _r\zeta  \big\rangle\d s
-\int_0^t\Big\langle \frac{h^\epsilon}{r} ,\frac{\zeta}{r}\Big\rangle\d s
+\int_0^t\Big\langle\Big(\frac{2h -\sin 2h}{2r^2}\Big)^\epsilon,\zeta  \Big\rangle\d s
\\
+\sum_{k\in\N}\int_0^t\langle\phi _k^\epsilon ,\zeta \rangle \d B_k(s)
+\frac12\sum_{k\in\N}\int_0^t\langle\varphi ''(h^\epsilon (s))(\phi _k^\epsilon )^2,\zeta \rangle\d s\,.
\end{multline}

At this step, the conclusion involves essentially the same arguments as that of \cite{debussche2016degenerate}, hence we only sketch the proof here (we refer to the latter reference for details).
For $f\in H\equiv L^2(\I;r\d r),$ there hold the basic properties (see e.g.\ \cite{evans2010partial}):

\begin{multline}
\label{convergence_convolution}
|f^\epsilon |_{H}\leq |f|_{H}\,,\quad |f^\epsilon -f|_{H}\to0\,,
\\
\text{and if}\enskip f\in D(A)\enskip \text{then}\enskip  
|A(f^\epsilon) |_{H}\leq |Af|_{H}\,,\quad |A(f^\epsilon -f)|_{H}\to0\,.
\end{multline}
Interpolating, we also have
\[
|f^\epsilon |_{V_\beta }\leq |f|_{V_\beta }\,,\quad |f^\epsilon -f|_{V_\beta }\to0\,,
\]
for every $\beta \in(4/3,2].$
Hence, each term in the drift of \eqref{1D-Ito} converges to the corresponding term of $h$ (note that for $\beta >4/3$ it has been already seen that $|b(r,h(r))|_{H}<\infty$, and recall that $|\partial _rf|_{H}+|\frac{f}{r}|_{H}\simeq|f|_{V_{1/2}}$), and similarly for the left hand side. 

Concerning the stochastic term, we have, using Burkholder-Davis-Gundy inequality:
\begin{multline*}
\E\sup_{t\leq T}\Big|\sum_{k\in \N}\int_0^t\left\langle \varphi '\big(h^\epsilon (s)\big)\phi _k^\epsilon -\varphi' \big(h(s)\big)\phi _k,\zeta \right\rangle\d B_k(s)\Big|
\\
\leq c\E\left(\int_0^T\sum_{k\in\N}\big|\left\langle \varphi '\big(h^\epsilon (s)\big)\phi ^\epsilon _k-\varphi '\big(h(s)\big)\phi _k,\zeta \right\rangle\big|^2\d s\right)^{1/2}
\\
\leq c'\E\left(\int_0^T|\varphi '(h^\epsilon )-\varphi '(h)|^2_{H}|\phi ^\epsilon |^2_{\mathbb{L}_2(H,V)}\d s\right)^{1/2}
\\
+c'\E\left(\int_0^T|\varphi '(h)|^2_{H}|\phi ^\epsilon-\phi  |^2_{\mathbb{L}_2(H,V)}\d s\right)^{1/2}\,,
\end{multline*}
where $\phi ^\epsilon:=\sum_{k\geq 1}\phi _k^\epsilon \langle e_k,\cdot  \rangle. $
Since $\phi \in \mathbb L_2(L^2,V_{\beta'})\subset \mathbb L_2(L^2,V)$ for $\beta '>\beta >1$, dominated convergence and boundedness of $\varphi ''$ imply the desired convergence.
\end{proof}
\begin{proof}[Proof of Theorem \ref{thm:loc_solv}]
\textit{Existence.}
Consider $h=h(t,r)$ as in Claim \ref{clm:LocalSolvability-h}, and fix $\theta \in[0,\pi).$
According to the above discussion, since the map 
\[
\R\ni h\mapsto \varphi(h):=u_h\equiv \left(\cos\theta \sin h,\sin\theta \sin h,\cos h\right)\in\R^3
\]
has bounded derivatives up to second order, then
for $\psi\equiv(\psi ^1,\psi ^2,\psi ^3)$ in $C^\infty_0(\D;\R^3),$ we can apply \eqref{ito_formula} to 
\[I\ni r\mapsto\zeta^i _\theta (r):=\psi^i (r\cos\theta ,r\sin\theta )\,,
\quad\text{for} \enskip 1\leq i\leq 3\,,
\]
provided $\psi $ vanishes at the origin. However, reasoning as above, we can give meaning to the formula for a general $\psi \in C^\infty_0,$ if one replaces $h$ by its $1-$fattening $\tilde h.$ More precisely, decomposing each map $\varphi '(h)^i\zeta ^i_\theta $ as in \eqref{decomposition_zeta} (leaving aside smoothness which is not crucial here), it is licit, as in \eqref{cauchy_principal}, to define
\[
\iint_{[0,t]\times [-1,1]}\nabla_r \tilde h\cdot \nabla_r (\varphi^i (\tilde h)\zeta ^i_\theta  )r\d r\d s :=\lim_{\epsilon \to0}\iint_{[0,t]\times [-1,\epsilon ]\cup[\epsilon ,1]}\nabla_r \tilde h\cdot \nabla_r (\varphi^i (\tilde h)\zeta_\theta^i  )r\d r\d s\,.
\]
Hence, observing that $\varphi '(h)^i=u_{h+\pi/2}^i\equiv\Phi _h^i,$ and summing over $i=1,2,3,$ one obtains:
\begin{multline*}
\int_{[-1,1]} (u_{\tilde h}(t)-u_{\tilde h_0})\cdot \psi  r\d r
\\
= -\iint_{[0,t]\times [-1,1]}\nabla_r \tilde h\cdot \nabla_r (\Phi _{\tilde h}\cdot \psi   )r\d r\d s
+\iint_{[0,t]\times [-1,1]}\frac{2\tilde h-\sin2\tilde h}{2r^2} u_{\tilde h}\cdot \psi   r\d r\d s
\\
-\frac12\sum_{k\geq 1}\iint_{[0,t]\times[-1,1]}u_{\tilde h}(\phi e_k)^2\psi r\d r\d s+\sum_{k\geq 1}\iint_{[0,t]\times[-1,1]}\Phi _{\tilde h}\cdot \psi(\widetilde{\phi e_k})  r\d r\d B_k\,,
\end{multline*}
(here the integrand $\psi $ is evaluated at $(r\cos\theta ,r\sin\theta )$).
Integrating over $\theta \in[0,\pi],$ changing the variables,
we end up with
\begin{multline}\label{eq_0}
\int_{\D}(u_h(t)-u_{h_0})\cdot \psi\d x
=\int_0^t\Big\langle Ah +b_{h}(s,|x|), \Phi _h \psi\Big\rangle_{W^{-1,2},W^{1,2}}\d s
\\
-\frac12\sum_{k\geq 1}\int_0^t\left\langle u_h(\phi e_k)^2,\psi \right\rangle_{L^2}\d s+\int_0^t\left\langle\Phi _{h} \d w_\phi,\psi \right\rangle_{L^2}\,,
\end{multline}
where $\langle \cdot ,\cdot \rangle_{W^{-1,2},W^{1,2}}$ denotes the dual pairing on the disk $\D$, this formula being justified by the fact that
\begin{multline*}
|\nabla (\Phi _h\cdot \psi )|_{L^2}^2 \equiv \iint\limits_{I\times[0,2\pi]}\Big[-\partial _rh(u_h\cdot \psi) 
\\
+\Phi _h\cdot \Big((\cos\theta-\sin\theta ) \partial _1\psi +(\sin\theta+\cos\theta )\partial _2\psi \Big) +\frac{\Theta \cdot \psi }{r}\Big]^2r\d r\d \theta <\infty
\end{multline*}
(again, $\psi $ is evaluated at $(r\cos\theta ,r\sin\theta )$ and we define $\Theta =(-\cos\theta ,\sin\theta ,0)$).
Furthermore, direct computations show that in the sense of distributions in $\D,$ we have the identity
\begin{equation}\label{id:Delta}
\Delta u_h+u_h|\nabla u_h|^2 =\left(\partial _{rr}h+\frac{\partial _rh}{|x|}-\frac{\sin2h}{|x|^2}\right)\Phi _h\,,
\end{equation}
Hence, we obtain that $u$ verifies \ref{def_ii}-\ref{def_iii} in Definition \ref{def:weak_sol}. The property \ref{def_i} is trivial.

Considering, for $t\in[0,\tau ^\beta ),$ the Fourier-Bessel series
\[
h^N(t,\cdot ):=\sum_{k=1}^N\langle h(t),e_k \rangle e_k(\cdot )\,,
\]
by the definition \eqref{nota:ek}, we observe that $h^N(t,1)=0,$ for every $N\geq 1,$ independently of the time-variable.
By the Sobolev embedding $V_{\alpha }\subset L^\infty,$ when $\alpha >1$ (see Remark \ref{rem:Sobolev_embedding}) the property remains also true for $h\equiv\lim_{N\to\infty,V_\beta } h^N.$  This shows the property \ref{def_iv}.

\item[\indent\textit{Uniqueness.}]
Conversely,
let $u(t,x)\equiv u_{h(t,|x|)}$ be a $1-$corotational map such that
\eqref{eq:weak_u} is fulfilled for every smooth $\psi$ with compact support in $\D.$ Consider $\psi(x):=\zeta (|x|) $ where $\zeta $ has compact support in $(0,1),$
and write  $\check w_\phi(t,x),$ resp.\ $\check h(t,x)$ instead of $w_\phi (t,|x|),$ resp.\ $h(t,|x|).$
Using e.g.\ the It\^o Formula
\footnote{in fact there holds $\d \,(u\cdot \Phi _{\check h})=\d u\cdot \circ\Phi _h +u\cdot \circ \d \,(\Phi _h)$ in $W^{-1,2};$ to see this apply Theorem 4.2.5.\ to $|u|_{L^2}^2,$ and $|\Phi _h|_{L^2}^2,$ and then polarize.
}
given in \cite{prevot2007concise}, since $u\perp\Phi _h,$
there holds:
\[
0\equiv \d \Big(\left\langle u,\Phi _{\check h} \psi \right\rangle_{L^2}\Big)
=\big\langle \d u,\circ\,\Phi _{\check h} \psi \big\rangle_{W^{-1,2},W^{1,2}_0}
+\left\langle u,\circ\d \big(\Phi _{\check h}\big) \psi \right\rangle_{W^{1,2}_0,W^{-1,2}}
\]
(differential sense).
Hence, by \cite[Prop.~A.1.]{debussche2016degenerate},
one has
\[
\begin{aligned}
0&=\Big\langle(\Delta u +u|\nabla u|^2)\d t +\Phi _{\check h}\circ\d \check w_\phi,\Phi _{\check h}\psi \Big\rangle_{W^{-1,2},W^{2,1}_0}
-\left\langle \d \check h,\circ(u\cdot u)\psi  \right\rangle_{W^{-1,2},W^{1,2}_0}
\\
&=\Big\langle\partial _{rr}\check h+\frac{\partial _r\check h}{|x|}-\frac{\sin2\check h}{|x|^2},\psi\Big\rangle_{W^{-1,2},W_0^{1,2}} \d t
+\left\langle\d \check w_{\phi },\psi \right\rangle_{L^2}
-\langle\d \check h,\psi \rangle_{W^{-1,2},W_0^{1,2}} \end{aligned}
\]
where we have used \eqref{id:Delta}, $u\cdot u\equiv|u|^2=1,$
and also the fact that, according to the It\^o Formula in \cite{prevot2007concise}, there holds for any $\varphi $ in $W^{1,2}_0(\D):$
\[
\d \left(\left\langle\Phi _{\check h},\varphi\right\rangle\right)
=\left\langle \frac{\d \Phi_x}{\d x}\Big|_{x=h}\circ\left(\d \check h\right),\varphi \right\rangle_{W^{-1,2},W^{1,2}_0}
\equiv\left\langle-u\circ\d \check h,\varphi \right\rangle_{W^{-1,2},W^{1,2}_0}\,.
\]
Changing the variables in the latter integrals, it follows that $h$ is a weak solution of \eqref{eq:h_dW_functional}, in the sense that it fulfills
\eqref{weak_solution}.
Hence, by the fact that weak solutions are also mild solutions (see \cite[Prop. 6.3]{DPZ}),
uniqueness follows from Claim \ref{clm:LocalSolvability-h} and Remark \ref{rem:equiv_A_Delta}.
\end{proof}

%
%
%
%
\section{Proof of Theorem \ref{thm:main}}\label{sec:proof_main_thm}
\subsection{Preliminary material and key proposition}
Fix $4>\beta>2.$
According to Remark \ref{rem:equiv_A_Delta}, we can focus on the proof of blow-up for $h,$ i.e.\ the colatitude of $u$ given by Claim \ref{clm:LocalSolvability-h}.
In the sequel, when $(h_0,z)\in V_\beta \times C([0,\infty);V_\beta )$, we will systematically denote by
$(h(h_0,z),\tau^\beta (h_0,z))$
the mild solution of \eqref{eq:h_dW_functional} on $\{Z=z\}$, and its maximal time of existence in $V_\beta$,
namely:
\begin{equation}\label{nota:h}
\left[\begin{aligned}
&h(h_0,z):=v+z\,,\enskip\\
&\text{where}\enskip v=v(h_0,z)\enskip \text{solves in the mild sense:}\\
&\hspace{1em}\begin{cases}
\hspace{0.2em}\partial _tv=Av+b_{v+z}\enskip \text{on}\enskip \big[0,\tau^\beta (h_0,z)\big)\times \I\,,\\
\enskip v|_{t=0}=h_0\,,
\end{cases}\\
&\text{and where}\enskip \tau^\beta (h_0,z) <\infty\enskip \text{implies}\enskip  \limsup_{t\to\tau^\beta (h_0,z) }|h(h_0,z,t)|_{V_\beta }=\infty\,,
\end{aligned}\right.
\end{equation}
(see Section \ref{sec:proof:solvability}).
Our approach is to show first that  a fixed $z\in C([0,2t_*];V_{\beta})$ with $z(0)=0$, there exists a ``pre-blow-up set'' $\mathfrak H_z,$ namely a set of initial conditions $h_0$ such that the associated solution $h(h_0,z,\cdot )$ blows up before $t_*$.
\begin{proposition}
\label{clm:blowup_z_fixed}
Let $4>\beta>2$, and fix $t_*>0$. There exists
$\bar\eta>0$, such that for all $z\in C([0,2t_*];V_{\beta})$ with $\left\|z\right\|_{2t_*,\beta}\leq \bar\eta$, there exists a parabola $\chi_*=\chi_*(z)$ belonging to the family \eqref{nota:chi_c}, and satisfying the property that: if 
$h_0\in V_\beta$ with $h_0\geq \chi_*$, then
\[
\tau^\beta(h_0,z)\leq t_*\,.
\]
Moreover, the pre-blow-up set $\mathfrak{H}=\{h_0\in V_\beta\,,\enskip h_0\geq \chi_*\}$, has \emph{nonempty interior} in $V_\beta$.
\end{proposition}
The following will be obtained in Subsection \ref{subsec:globalization}
through a topological argument.
\begin{corollary}
\label{mainProp}
Let $4>\beta>2$.
For any $t_*>0$, there exist two subsets $\mathfrak{H}$ of $V_\beta$, $\mathfrak{Z}$ of $C([0,t_*];V_\beta)$, with \emph{nonempty interiors},
such that for all $(h_0,z)\in\mathfrak H\times\mathfrak Z\,$:
\[
\tau^\beta(h_0,z)\leq t_*\,.
\]
\end{corollary}

A few preliminaries and notations are now needed to prepare the proof of Proposition \ref{mainProp}. As in Chang-Ding-Ye's proof, we will make use of a comparison principle for the scalar parabolic equation \eqref{eq:h_classical}.
It is however different from that of \cite{chang1992finite,bertsch2002nonuniqueness}, because the nonlinearity depends on the realization of the Ornstein Uhlenbeck process $Z:\Omega\to C([0,T];V_\beta )$.
However additiveness of the noise in \eqref{eq:h_dW_classical} allows to appeal to deterministic theory only, fixing $\omega \in\Omega$ and letting $z:= Z(\omega )$.

We consider equations of the form
\begin{equation}\label{eq:gene}
\partial _tf=Af - \frac{p\big(z(t,r)+f(t,r)\big)}{r^2}\, ,\quad\text{for}\enskip (t,r)\in[0,\kappa ]\times \I\setminus\{0\}\,,
\end{equation}
where $p:\R\to\R$ vanishes at the origin and
\begin{equation}\label{ass_z}
z\in C([0,\kappa ];V_\beta)\,\enskip \text{for some}\enskip \beta >1\enskip \text{and}\enskip z|_{\{0\}\times \I}=0
\end{equation}
(for such $z,$ using the definition of the scale $(V_\beta ),$ we have in fact $z|_{\Sigma}=0$).
In order to take into account the main cases we have in mind,
we will assume that the nonlinearity fulfills the following properties.
\begin{assumptions}
We will assume that $p:\R\to \R$ is of class $C^2$ around the origin, and that
\begin{equation}\label{ass}
p(0)=0\,;\quad p'(0)>-1\,;
\quad |p(x)-p(y)|\leq K|x-y|\enskip,\enskip \forall x,y\in\R\,,
\end{equation}
for some universal constant $K>0$.
\end{assumptions}
\begin{remark}\label{rem:comparison}
Assumptions \eqref{ass} do cover the cases where: (a) $p(x)=0$ (comparison principle for the linear equation $(\partial _t-A)f=0$), and
(b) $p(x)=\sin (2x)/2-x$ (comparison principle for \eqref{eq:h_dW_functional}).
\end{remark}
The proof of the following result is postponed to \ref{app3}. Here $\J$ denotes any compact subinterval of $\I.$
\begin{comparison}[\ref{eq:gene}]
Fixing some $\beta >1$, assume that the assumptions \eqref{ass_z} and \eqref{ass} are fulfilled, and that we are given $f,g\in C([0,\kappa);V_\beta )\cap C^1([0,\kappa );H)$, such that
\begin{enumerate}[label=(\roman*)]
\item\label{ineq:df_pro:comparison_principle}
 $-\int_0^\kappa\int_{\J}f\partial_t\zeta r\d r\d t \leq -\int_0^\kappa \int_{\J}( \partial _rf\partial _r\zeta +\frac{f+p (z+f)}{r^2}\zeta )r\d r\d t \,$,
\item\label{ineq:dg_pro:comparison_principle}
$-\int_0^\kappa\int_{\J}g\partial_t\zeta r\d r\d t \geq -\int_0^\kappa \int_{\J}( \partial _rg\partial _r\zeta +\frac{g+p (z+g)}{r^2}\zeta )r\d r\d t\,$,
\end{enumerate}
for all \text{nonnegative} $\zeta\in C^\infty([0,\kappa]\times \J)$ with $\zeta(t,r)=0$ on the parabolic boundary $\Sigma:=\{0\}\times \J\cup[0,\kappa)\times\partial \J$.
Assume moreover that $f\leq g$ on $\Sigma$.
Then:
\[
\text{For almost every}\enskip (t,r)\in[0,\kappa)\times \J\,,\enskip \text{we have}\enskip   f(t,r)\leq g(t,r)\,.
\]
\end{comparison}

In the sequel, for $k>0$, and $r\in{\I}$, we denote by
\begin{equation}\label{nota:chi_c}
\chi_k(r):=\, k\left(2r-3r^3+r^5\right)=\,kr(1-r^2)(2-r^2)\,.
\end{equation}
see Fig.\ 
\ref{fig:plot-chi1chi2}.
\begin{figure}
\begin{center}
 \includegraphics[width=0.65\linewidth]{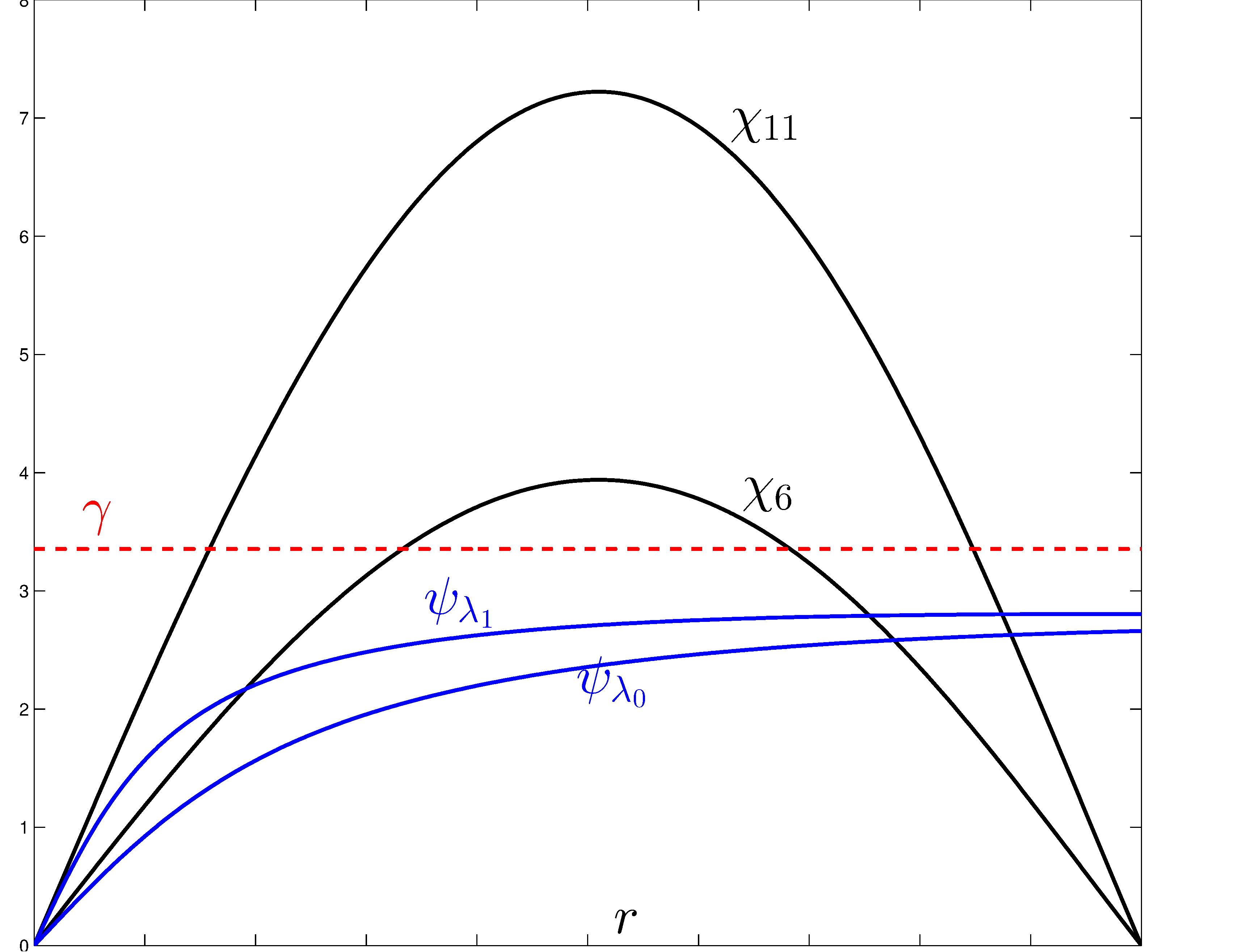}
 \caption{Plots of $\psi_{\lambda _0},\psi_{\lambda _1}$ for some $\lambda _1<\lambda _0$, together with $\chi _6,\chi _{11}$ and $\gamma $, for $r\in \I\equiv[0,1]$.}
 \label{fig:plot-chi1chi2}
\end{center}
\end{figure}
Initial data $h_0$ will be compared according to their position with respect to the reference family $(\chi_k )_{k\in\N}$, see the subsections below.
The choice of such parabolae rules out the pathological case where $\tau^\beta (\chi_k,z)$ equals $0$. Indeed,
on the one hand we have $A^2\chi _k(r)=48kr$, $r\in \I\setminus\{0\}$, which belongs to $V_{1/2-\epsilon}$ for any $\epsilon>0$ (this fact is left to the reader).
On the other hand:
$\chi_k(\partial \I)\equiv0$, and $A\chi_k(\partial \I)\equiv0$,
which ensures that for $k\in\R$:
\begin{equation}\label{chi_92}
\chi_k\in V_{\alpha }\,,\quad \text{for any}\enskip \alpha <\frac92\,.
\end{equation}

\subsection{First step in the proof of Theorem \ref{thm:main}: proof of Proposition  \ref{mainProp}}

In the following lemma, we exhibit an explicit family of maps $\{\psi_{\epsilon ,\mu ,\lambda _0,\xi} \}$ satisfying the differential inequality
\begin{equation}\label{diff_ineq}
\partial_t\psi\leq A\psi +b_{\psi+z}\enskip,
\end{equation}
up to some positive time.
Note that in the sequel, since $\beta$ is strictly bigger than $1$, we restrict our attention to maps that have a continuous version (see Remark \ref{rem:Sobolev_embedding}). Hence, when $f,g\in V_\beta $ we may write $f\leq g$ for $\langle f,\zeta \rangle\leq \langle g,\zeta  \rangle$ for all non-negative $\zeta \in C^\infty_0(\I).$

Here we denote by $\J:=[0,r_1]$ for some $r_1\in(0,1)$.
\begin{lemma}\label{lem:subsolution_f}
Fix $z$ in $C([0,\infty);V_\beta)$ with $z(0)=0$.
For $\lambda_0,\epsilon,\delta>0$, define
$\lambda: t\in[0,T_{\lambda_0})\mapsto \lambda(t)$ as the solution of the ODE :
\begin{equation}\label{nota:lambda}
 \lambda'=-\delta\lambda^\epsilon,\quad 0<t\leq T_{\lambda_0}:=\frac{\lambda_0^{1-\epsilon}}{(1-\epsilon)\delta}\enskip,\enskip \text{with initial data}\enskip 
 \lambda(0)=\lambda_0\,.
\end{equation}
Assume that there exist $t_+>0$, $\xi\in V_\beta$, $\xi \geq 0$, depending on $z$,
such that
\begin{equation}\label{nota_xt}
x(t,r):=S(t)\xi(r)+z(t,r)\geq 0\quad \text{for}\enskip  t\in[0,t_+]\enskip \text{and}\enskip r\in \J\,,
\end{equation}
where $S(t)=e^{tA}$, see \eqref{nota:A}.

Fix $0<\epsilon<1$. There exist positive constants $\bar\mu(\epsilon)$, $\bar\delta(\epsilon)$, such that for all
$\mu\geq\bar\mu(\epsilon )$ and $0<\delta\leq \bar\delta(\epsilon )$,
for all $\lambda_0>0$ defining $\lambda=\lambda_{\epsilon,\delta,\lambda_0}(t)$ as in \eqref{nota:lambda}, then the map
\begin{equation}\label{ansatz_psi}
 \psi(r,t)=\arccos\left(\frac{\lambda(t)^2-r^2}{\lambda(t)^2 + r^2}\right) +\arccos\left(\frac{\mu^2-r^{2+2\epsilon}}{\mu^2+r^{2+2\epsilon}}\right) +S(t)\xi(r)\enskip,
\end{equation}
fulfills the differential inequality \eqref{diff_ineq} on $[0,t_+\wedge T_{\lambda _0}]\times \J$.
\end{lemma}
\begin{proof}[Proof of Lemma \ref{lem:subsolution_f}]
Let $0<\epsilon<1$.
As in \cite{chang1992finite}, we set for $(\lambda,r)\in\R_+^*\times{\J}$:
\begin{equation}\label{nota:phi_theta}
\varphi_{\lambda}(r):=\arccos\left(\frac{\lambda^2-r^2}{\lambda^2 + r^2}\right),\quad \theta_{\epsilon,\mu}(r):=\arccos\left(\frac{\mu^2-r^{2+2\epsilon}}{\mu^2+r^{2+2\epsilon}}\right)\,.
\end{equation}
Recall that for any fixed triplet $\lambda,\epsilon,\mu>0$, the maps $\varphi_\lambda\,$, $\theta_{\epsilon,\mu}$ satisfy for  $r\in \J$ (see \cite{chang1992finite}):
\begin{equation}
\label{eq:theta}
\begin{aligned}
&A\varphi_\lambda(r)=\frac{\sin2\varphi_{\lambda}(r)-2\varphi_{\lambda}(r)}{2r^2}\\
&A\theta_{\epsilon,\mu}(r)=\frac{(1+\epsilon)^2\sin2\theta_{\epsilon,\mu}(r)-2\theta_{\epsilon,\mu}(r)}{2r^2}
\end{aligned}
\end{equation}
Now, since $\theta_{\epsilon,\mu}(r)\to0$ as $\mu\to\infty$, it is possible to choose a parameter $\bar\mu(\epsilon)$,
such that for all $r\in{\J}$,
\begin{equation}\label{nota:mu_bar}
\cos\theta_{\epsilon,\mu}(r)\geq\dfrac{1}{1+\epsilon}\,.
\end{equation}
Take $\mu\geq \bar\mu$, and let $z\in C([0,t_+];V_{\beta})$ be such that $x=S\xi+z$ takes nonnegative values on $[0,t_+]\times{\J}$.
For $t\in[0,T_{\lambda_0})$, $r\in{\J}$, define
$\psi(t,r):=\varphi_{\lambda(t)}(r)+\theta(r)+S(t)\xi(r)$,
and denote by $\theta:=\theta_{\epsilon,\mu}(\cdot )$, $\varphi:=\varphi_{\lambda(\cdot )}(\cdot )$, and $S\xi:=t\in\R_+\mapsto S(t)\xi$.
On the one side, using \eqref{eq:theta}, the trigonometric identities $\sin2\theta=2\cos\theta\sin\theta$ and $\sin2\varphi-\sin2(\varphi+\theta)=-2\sin\theta\cos(2\varphi+\theta)$,
there comes
\begin{equation}\label{calcul_1}
\begin{aligned}
A\psi+b_{\psi+z}&=A(\varphi+\theta+S\xi)+b_{\varphi+\theta+x}\\
&=\frac{1}{2r^2}\big[(1+\epsilon)^2\sin2\theta+\sin2\varphi-\sin2(\varphi+\theta)\\
& \hspace{6em}+ 2x +\sin2(\varphi+\theta)-\sin2(\varphi+\theta+x)\big]+AS\xi\\
&=\frac{1}{2r^2}\big[2(1+\epsilon)^2\sin\theta\cos\theta-2\sin\theta\cos(2\varphi+\theta)+F_{\varphi,\theta}(x)\big]+AS\xi
\end{aligned}
\end{equation}
where we denote by
$F_{\varphi,\theta}(x)=2x-\big(\sin2(\varphi+\theta+x)-\sin2(\varphi+\theta)\big)\,,\,x\in\R\,$, for $\varphi,\theta\in\R$.
Using \eqref{nota:mu_bar}, the right hand side in \eqref{calcul_1} is bounded below by
$1/(2r^2)[(2+2\epsilon-2\cos(2\varphi+\theta))\sin\theta  + F_{\varphi,\theta}(x)]+AS\xi$,
so that
$A\psi+b(r,\psi+z)\geq 1/(2r^2)[2\epsilon\sin\theta+F_{\varphi,\theta}(x)]+AS\xi$.
Expanding $\sin\theta $, we eventually obtain
\[\begin{aligned}
A\psi +b(r,\psi+z)(t,r)&\geq \frac{2\epsilon \mu r^{\epsilon-1}}{\mu^2+r^{2(1+\epsilon)}}+\frac{F_{\varphi,\theta}(x(t,r))}{r^2}+AS(t)\xi(r)\\
&\geq \frac{2\epsilon\mu r^{\epsilon-1}}{\mu^2+1}+\frac{F_{\varphi,\theta}(x(t,r))}{r^2}+AS(t)\xi(r)\,,
\end{aligned}\]
a.e.\ on $[0,t_+]\times \J\setminus\{0\}$ (and therefore in the sense of positive test functions).
Now, regardless of the values taken by the parameters $\theta,\varphi$, the map $F_{\varphi,\theta}$ vanishes at the origin, and has nonnegative derivative on $\R_+$. We deduce that since $x\geq 0$ on $[0,t_+]\times \J$, then so is $F_{\varphi,\theta}(x)$.
Moreover, simple computations show that for $r\in \J$:
\[
\partial_t\psi (t,r)=\frac{2\delta\lambda(t)^\epsilon r}{\lambda(t)^2+r^2}+AS(t)\xi(r)\enskip,
\]
Thus, if for every $(t,r)$ in $[0,t_+]\times \J\,,$ we have
$2\epsilon\mu r^{\epsilon-1}/(\mu^2+1)\geq 2\delta\lambda(t)^\epsilon r/(\lambda(t)^2+r^2)$,
then $\partial_t\psi\leq A\psi +b(r,\psi+z)$ holds. Setting $s=r/\lambda(t)$, it is however sufficient to verify that:
$\sup_{s\in\R}\frac{s^{2-\epsilon}}{1+s^2}\leq \frac{\mu\epsilon}{\delta(\mu^2+1)}$,
which is true if $\delta\geq \bar\delta$ for $\bar\mu>0$ as in \eqref{nota:mu_bar}.
This proves the lemma.
\end{proof}
\begin{remark}\label{rem:blow_up}
Let $\beta >2$, and consider $\xi,z,t_+$ as in Lemma \ref{lem:subsolution_f}, and assume that
\begin{equation}\label{hyp_t}
t_+\geq T_{\lambda _0}\equiv\frac{\lambda_0^{1-\epsilon}}{(1-\epsilon)\delta}\,.
\end{equation}
Then, the map $f:=\psi_{\epsilon,\mu,\lambda_0,\xi}$ constructed above blows up at $t=T_{\lambda_0}$.
Indeed, since $\xi\in V_\beta$ with $\beta>2$, then $|{\partial_rS(t)\xi}|_{L^{\infty}}$ is bounded for $t\in[0,T_{\lambda_0})$ (see Remark \ref{rem:Sobolev_embedding}), and:
\[\partial_r \psi(t,0)=\partial_r\varphi(t,0)+\theta'(0)+ \partial_rS(t)\xi(0)
=\frac{2}{\lambda(t)}+\partial_rS(t)\xi(0)\enskip \underset{t\to T_{\lambda_0}}{\longrightarrow}\infty\,.
\]
Let $(h,\tau):=(h(h_0,z),\tau^\beta(h_0,z))$ be the mild solution defined in \eqref{nota:h}, where $h_0\in V_\beta$. Assume $\tau\geq T_{\lambda_0}$, and that:
\begin{equation}\label{ineq_boundary}
f\leq g\quad\text{on}\enskip \{0\}\times \J \cup[0,T_{\lambda_0})\times\partial \J\,,
\end{equation}
where we let $g:=h-z$, and $\J:=[0,r_1]$ for some $r_1\in(0,1)$.

Up to a straightforward extension of $f$ on the whole space interval $I$, observe that $f\in C([0,T];V_\beta )\cap C^1([0,T];H)$ (this can be shown by direct computations). On the other hand, recall that $g\equiv h(z)-z$ satisfies $g(t)=S(t)h_0+\int_0^tS(t-s)b_{g+z}\d s$ for $t\leq\tau $.
Since $\beta >4/3$, then \eqref{P2} yields $b_{g+z}\in C([0,T];H)$, so that
\begin{equation}\label{regularity_g}
g\in C([0,T];V_\beta )\cap C^1([0,T];H)\,.
\end{equation}
By \eqref{ineq_boundary}, \eqref{regularity_g} and Lemma \ref{lem:subsolution_f}, the comparison principle for \eqref{perturbedEq} can be applied so that
$f\leq g$ on $[0,T_{\lambda_0})\times \J$.
Since the maps $f,g$ vanish at the origin regardless of the time variable, it follows that:
\[
\partial_rf(t,0)\leq\partial_rg(t,0)\,,\enskip \text{for all} \enskip t\in[0,T_{\lambda_0})\,,
\]
and then $|\partial_r h|_{L^{\infty}}\to\infty$ as $t\to T_{\lambda_0}$, which by Remark \ref{rem:Sobolev_embedding}
implies blow-up also in the sense that $\limsup_{t\to T_{\lambda _0}}|h(t)|_{\beta }=\infty$.
\end{remark}
We can now turn to the proof of Proposition \ref{clm:blowup_z_fixed}.
\begin{proof}[Proof of Proposition \ref{clm:blowup_z_fixed}]
Fix $2<\beta<4$.
For each $z\in C([0,2t_*];V_{\beta})$, and $\xi\in V_\beta$, we define $x=x_{\xi ,z}$ by:
\begin{equation}\label{def_x}
x(t)=S(t)\xi+z(t)\,,\enskip\text{for}\enskip t\leq 2t_*\,.
\end{equation}
In what follows we denote by $\J$ the compact interval $[0,1/2]$.

\item[\indent\textit{Step 1: nonnegativeness of $x$ up to a positive time.}]
Assume that $\xi\geq \chi_1$ on $\J$, where $\chi_1$ is the parabola defined by \eqref{nota:chi_c}. Note that such $\xi \in V_\beta $ exists for $\beta >2$ since is suffices to let for instance $\xi :=\chi _1$, see \eqref{chi_92}.
Our aim now is to show that if the perturbation $z$ is not too large in $C([0,2t_*];V_{\beta})$, then the map $x$ defined above stays nonnegative on $\J$.

We first claim that there exists a constant $\eta>0$, such that
for all $\xi ,y\in V_\beta$ with $\xi \geq \chi _1$,
\begin{equation}\label{nota:eta0}
|\xi-y|_{\beta}\leq 2\eta\enskip\Rightarrow \enskip {y}|_{\J}\geq0\,.
\end{equation}
Indeed, since $\beta>2$, then there exists $c_\beta>0$, such that for all $y\in V_\beta$, (see Remark \ref{rem:Sobolev_embedding}),
\[|\partial_r\xi-\partial_ry|_{L^{\infty}({\J})}\leq c_\beta|\xi-y|_{\beta}\,.\]
Choose $\eta=c/(2c_\beta)$, where $c$ is such that $\chi_1(r)-cr\geq 0$ for $r\in{\J}$ (note that $c$ and therefore $\eta$ do not depend on $\xi$), so that $|y-\xi|_{\beta}\leq \eta$ will imply
$|\partial_ry-\partial_r\xi|_{L^{\infty}}\leq c/2$. We conclude by the Mean Value Theorem,
observing first that both maps equal zero at the origin: if $\left|\xi-y\right|_{\beta}\leq \eta$,
then $\forall r\in{\J}$, $y(r)\geq \xi(r)-cr\geq \chi_1(r)-cr$ and thus $y(r)\geq 0$,
which proves \eqref{nota:eta0}.

Furthermore, for a fixed $\xi\in V_\beta$ with $\xi\geq \chi_1$, since $S$ is a strongly continuous semigroup,
there exists $t_+(\xi)>0$ such that
\[\text{for all}\enskip t\in[0,t_+(\xi)]\,,\quad|S(t)\xi-\xi|_{\beta}\leq \eta\,,\]
and thus for $0\leq t\leq t_+(\xi)$, $\|z\|_{2t_*,\beta }\leq\eta$, the map $x$ defined in \eqref{def_x} verifies
\begin{equation}\label{ineq:x0}
|x(t)-\xi|_{\beta}\leq |S(t)\xi-\xi|_{\beta}+|z(t)|_{\beta}\leq 2\eta \,.
\end{equation}
We have to get rid of the dependence of $t_+(\xi)$ with respect to $\xi$.
But if $\xi\in V_\beta$ with $\xi\geq \chi_1$ on $\I$, apply the linear comparison principle (see the previous subsection) on the whole interval $\I$ to $f := S(\cdot )\chi_1$, $g:=S(\cdot )\xi$, $\kappa :=2t_*$ (note that we have $f\leq g$ on $\{0\}\times \I\cup[0,2t_*]\times\partial \I$). We obtain that 
\[
t_+(\xi)\geq t_+(\chi_1)\,.
\]
Now define $t_+:=t_+(\chi_1)$. We have proven that 
there exists $\eta>0$ such that for all $t\in[0,t_+]$,
for all $\xi\in V_\beta$ with $\xi\geq \chi_1$ on ${\I}$, for all $z\in C([0,2t_*];V_{\beta})$
with $\left\|z\right\|_{2t_*,\beta}\leq \eta$, then
\begin{equation}\label{ineq_x}
x|_{[0,t_*]\times \J}\geq 0\,.
\end{equation}
\item[\indent\textit{Step 2. Construction of a pre-blow-up set for a fixed $z$.}]
Once and for all, fix $\eta $ as in Step 1, $z\in C([0,2t_*];V_\beta )$ with $\|z\|_{2t_*,\beta }\leq \eta $, and
$\xi \in V_\beta $ with $\xi \geq \chi _1$ on $\J$, so that \eqref{ineq_x} holds for $x=x_{\xi ,z}$.

It suffices to prove the proposition with $t_*\wedge t_+$ instead of $t_*$.
Therefore, without loss of generality we assume in the sequel that 
\[
t_+=t_*\,.
\]
In order to lighten the notations we also
denote by $\tau=\tau ^\beta(\cdot ,z)$, and $h=h(\cdot ,z)$.
Take any $0<\epsilon<1$, and fix $\mu\geq\bar\mu(\epsilon)$, $\delta\leq \bar\delta(\epsilon)$ and $\lambda=\lambda _{\epsilon ,\delta ,\lambda _0}(t)$ as in Lemma \ref{lem:subsolution_f}, where $\lambda _0>0$ is chosen such that
\[
T_{\lambda _0}=\frac{\lambda_0^{1-\epsilon}}{\delta(1-\epsilon)}\leq t_+\,,
\]
so that we know by Lemma \ref{lem:subsolution_f}, that the map $f_0:=\psi_{\epsilon ,\mu ,\lambda_0,\xi}$ defined as
in \eqref{ansatz_psi}, fulfills
\begin{equation}\label{diff_ineq_2}
\partial_tf_0\leq Af_0 +b(r,f_0+z)\enskip \text{on}\enskip [0,T_{\lambda _0})\times \J\,,
\end{equation}
with blow-up at $t=T_{\lambda _0}$.
Our strategy is to take $h_0\geq f_0|_{t=0}$, compare $g:=h(h_0,z,\cdot )-z$ with this ansatz, and then conclude by Remark \ref{rem:blow_up} that blow-up of $h$ happens before $t_+$.
For that purpose it remains however to choose $h_0$ in such a way that
\eqref{ineq_boundary} holds.
Note that if $h_0\in V_\beta $ is taken such that
\begin{equation}\label{condition_sup}
\left(h(h_0)-z\right)|_{[0,t_+]\times\{\frac12\}}> \sup\limits_{(r,\lambda)\in{\J}\times\R_+^*}\big(\varphi_{\lambda}(r)+\theta(r)+S(t)\xi(r)\big)\,,
\end{equation}
then we will have
\[(h(h_0)-z)|_{[0,t_+]\times \{\frac12\}}\geq \psi_{\epsilon ,\mu ,\lambda_0 ,\xi}\Big(t,\frac12\Big)\,,\]
with $\psi$ as in \eqref{ansatz_psi}, and this will hold regardless of $\epsilon ,\mu ,\lambda _0,$ and $0\leq t\leq t_+ $. In particular \eqref{condition_sup} will imply the bound needed on $[0,t_+]\times\partial \J$. Moreover, note that $\pi$ is an upper bound for the family of maps $(\varphi_{\lambda}(\cdot))_{\lambda>0}$ (see Figure \ref{fig:plot-chi1chi2}).
This motivates the following definition:
let
\begin{equation}\label{nota:gamma}
\gamma:=\pi+|\theta_{\epsilon,\mu}|_{L^{\infty}}+\sup_{t\geq0}|S(t)\xi|_{L^{\infty}}\enskip,
\end{equation}
and for $h_0\in V_\beta $, define
\begin{equation}\label{nota:t_Sigma_12}
t_{\Sigma}(h_0)=\inf\left\{0\leq t\leq \tau(h_0)\,,\enskip \left(h(h_0,t)-z(t)\right)|_{\{\frac12\}}\leq \gamma\right\}\,,
\end{equation}
with the understanding that $t_{\Sigma}(h_0)=\tau(h_0)$ if the set is empty.

\noindent Note that $\gamma$ is well-defined.
Indeed for any $u=\Sigma_k u_ke_k\in V_\beta$, by Remark \ref{rem:Sobolev_embedding}, since $\beta>1$,
the mapping $t\mapsto |S(t)u|_{L^{\infty}}=|\Sigma_ku_ke^{t\lambda _k}e_k |_{L^{\infty}},\enskip t\geq0$, is bounded (see \eqref{nota:ek}).

We claim now that there exists an integer $k=k(z)\geq 1$ such that for all $h_0\in V_\beta$, if $h_0\geq\chi_{k}$ on $I$, then
\begin{equation}\label{eq:t12_cr}
\tau(h_0)\leq t_+\,.
\end{equation}
Indeed, let $h_0\in V_\beta $ with $h_0|_{\J}\geq f_0|_{\{0\}\times \J}$
and assume that $\tau(h_0)>t_+$.
Note that necessarily
$\inf_{t\in[0,t_+]}\left(h(h_0)-z\right)|_{[0,t_+]\times\{\frac12\}}< \gamma$,
otherwise by comparison between $f _0$ and $g:=h(h_0)-z$, Remark \ref{rem:blow_up} would yield blow-up for $h(h_0)$ before $t_+$. So we have 
\begin{equation}\label{t_sigma_h0}
t_{\Sigma}(h_0)\leq t_+\,.
\end{equation}

Now, choose any $\lambda_1>0$ with $T_{\lambda_1}=\lambda_1^{1-\epsilon}/(\delta(1-\epsilon))\leq t_{\Sigma}(h_0)$,
and define $f_1:=\psi_{\epsilon ,\mu ,\lambda_1,\xi}$ by the formula \eqref{ansatz_psi} with $\lambda _1$ instead of $\lambda _0$.
Since $\arccos$ is Lipschitz out of $0$, we can always find $k\geq 1$ 
such that for $r\in \J$:
\begin{equation}\label{init_f1}
\chi_k(r)\geq f_1(0,r)\equiv\arccos\left(\frac{\lambda_1^2-r^2}{\lambda_1^2 + r^2}\right) +\arccos\left(\frac{\mu^2-r^{2+2\epsilon}}{\mu^2+r^{2+2\epsilon}}\right) +\xi(r)\,,
\end{equation}
where $\chi _k$ is as in \eqref{nota:chi_c}.
Consider any $h_1\in V_\beta$ with $h_1\geq \chi_k$ on ${\I}$.
One has the following alternative.
\begin{caseone}
In this situation, we have:
\begin{equation}\label{inequalities_either1}
 T_{\lambda_1}\leq t_{\Sigma}(h_1)\leq \tau(h_1)\,,\enskip \text{and}\enskip h(h_1)\geq f_1\enskip \text{on}\enskip \{0\}\times \J\cup[0,T_{\lambda _1}]\times \partial \J\,,
\end{equation}
and the comparison principle for \eqref{perturbedEq} can be applied with
$\kappa =T_{\lambda _1}$, $f:=f_1$ and $g:=h(h_1)-z$.
By Remark \ref{rem:blow_up} we obtain that $h$ blows-up before $T_{\lambda _1}$,
whence $ T_{\lambda_1}=\tau(h_1)\leq t_+$.
\end{caseone}
\begin{casetwo}
In this case,
apply the comparison principle for \eqref{perturbedEq} on the whole interval $\I$ with $\kappa:=\tau(h_0)\wedge \tau(h_1)$, $f:=h(h_0)-z$, and $g:=h(h_1)-z$,
so that in particular:
\begin{equation}\label{gamma_f}
\text{on}\enskip [0,\tau(h_0)\wedge \tau(h_1))\,,\enskip \text{there holds}\enskip f\Big(\cdot\, ,\,\frac12\Big)\leq g\Big(\cdot\, ,\,\frac12\Big)\,.
\end{equation}
Therefore, necessarily $\tau(h_1)=t_{\Sigma}(h_1)$, otherwise we would have
\[
g\Big(t_{\Sigma}(h_1),\frac12\Big)\equiv\gamma < f\Big(t_{\Sigma}(h_1),\frac12\Big)\,,
\]
contradicting \eqref{gamma_f}.
Moreover, one has $t_{\Sigma}(h_1)\leq t_+$, and thus $\tau(h_1)\leq t_+$.
\end{casetwo}
We see that in both cases \eqref{eq:t12_cr} is true, and the claim implies that
\[
\mathfrak{H}:=\{h_1\in V_\beta \,,\enskip h_1\geq \chi _{k(z)}\}
\]
defines a pre-blow-up set for the individual element $z$.
\item[\indent\textit{Step 3. Nonemptiness of $\interior{\mathfrak H}$}.]
It suffices to show the result when $k=1$, namely that the set
$\mathfrak{H}=\{h_1\in V_\beta \,,\enskip h_1\geq \chi _{1}\}$
has nonempty interior for the topology of $V_\beta $.
Set $h_0=\chi_{2}\in\mathfrak{H}$, so that $h_0\in\mathfrak{H}$.
By Remark \ref{rem:Sobolev_embedding}, since $\beta>2$, there exists a sufficiently small radius $R>0$ such that if $h_1\in V_\beta$ with $|h_1-h_0|_\beta$, then $|\partial_rh_1-\partial_rh_0|_{L^{\infty}}\leq 1/2$.
By the Mean Value Theorem, since $h_0$ and $h_1$ vanish for $r\in\{0,1\}$,
then for all $r\in[0,1/2]$:
$|h_1(r)-h_0(r)|\leq (1/2)r\leq r(1-r),$
and the same holds when $r\in[1/2,1]$.
Thus for a.e.\ $r\in{\I}$:
\[
|h_1(r)-h_0(r)|\leq r(1-r)\,.
\]
The reader may also check that
\[
\forall r\in{\I},\enskip \chi_1(r)=r(1-r^2)(2-r^2)\geq r(1-r)\,.
\]
Thus, for all $h_1$ belonging to an open ball centered at $h_0=\chi_2$, and for all $r\in{\I}$:
$h_1(r)\geq \chi_{2}(r)-cr(1-r)\geq\chi_{1}(r)$,
which means that $h_1\in\mathfrak{H}$.
This finishes the proof of Proposition \ref{clm:blowup_z_fixed}.
\end{proof}
\subsection{Second step in the proof of Theorem \ref{thm:main}: proof of Corollary \ref{mainProp}}
\label{subsec:globalization}
Fix $\bar\eta>0$ as in Proposition \ref{clm:blowup_z_fixed}.
So far, we have shown that given a trajectory $z$ in the ball $\mathbb{B}\subset C([0,2t_*];V_{\beta})$, centered at zero and of radius $\bar\eta$, there exist an integer $k(z)$, such that for all $h_0\in V_\beta$, $h_0\geq \chi_{k(z)}$, then $\tau^\beta(h_0,z)\leq t_*$.
To conclude we would need in some sense to ``reverse the quantifiers''.
It turns out that this can be achieved by a simple topological argument, the use of which seems to be new in the context of SPDEs (to the best of our knowledge).

Define
\[
\mathfrak{F}_k(t_*)=\{z\in \mathbb{B}\,,\enskip\forall h_0\in V_\beta\enskip \text{with}\enskip h_0\geq \chi_k\enskip \text{on}\enskip \I\,,\enskip  \tau^\beta(h_0,z)\leq t_*\}\,.
\]
We claim that $\mathfrak{F}_k(t_*)$ is a closed subset of $\mathbb{B}$.
Indeed, by definition:
if $z\in \mathbb{B}\setminus\mathfrak{F}_k(t_*)$, there exists $h_0\in V_\beta$ with $h_0\geq \chi_k$ on $\I$ and $\tau(h_0,z)>t_*$. Let $z^n\in\mathbb{B}\to z$ in $\mathbb{B}$, as $n\to\infty$. Let $\epsilon>0$ such that $h(h_0,z,\cdot )$ is defined on $[0,t_*+\epsilon]$. 
By Lemma \ref{lem:c_dep_result}, $h(h_0,z^n,\cdot )$ will be defined up to $t_*+\epsilon$,
provided $n$ is large enough. And thus $(\mathfrak{F}_k(t_*))^c$ is an open set of $\mathbb{B}$, which proves the claim.

By Proposition \ref{clm:blowup_z_fixed}, if $z\in\mathbb{B}$, then there exists $k$ such that $z\in\mathfrak{F}_k(t_*)$, thus 
\[
\mathbb{B}=\underset{k\in\mathbb{N}}{\bigcup}\mathfrak{F}_k(t_*)\,.
\]
Hence, by Baire's Theorem, there exists at least one $k^*$ such that $\mathfrak{F}_{k^*}(t_*)$ has non-empty interior. Thus we can set $\mathfrak{Z}=\mathfrak{F}_{k^*}(t_*)$. If we define $\mathfrak{H}=\{h_0\in V_\beta\,,\enskip h_0\geq \chi_{k^*}\}$, then for all $(h_0,z)\in\mathfrak H\times\mathfrak Z$, there holds $\tau(h_0,z)\leq t_*$. This finishes the proof of Proposition \ref{mainProp}.\hfill\qed

\subsection{End of the proof of Theorem \ref{thm:main} and closing remarks}
\label{subsec:proof_thm}
Given $T_1>0$ and $h_0, h_1\in V_\beta$, observe that there exists a control $z_1\in C([0,T_1];V_{\beta})$ with $z_1(0)=0$, such that:
$h(h_0,z_1,\cdot)$ exists on $[0,T_1]$, and
\begin{equation}\label{controlability}
h(h_0,z_1,T_1)=h_1\,.
\end{equation}

Indeed , similar to \cite{de2002effect}, 
we set for $t\in[0,T_1]$:
$\varphi (t):= (T_1-t)h_0/T_1+th_1/T_1$,
 and define $f(t):=(\varphi(t)- h_0) - \int_0^t\left(A\varphi(s)+b_{\varphi}(s)\right)\d s$.
Taking now
\[
z_1(t)=\int_0^tS(t-s)\frac{\d f}{\d s}\d s\,,\quad t\in[0,T_1]\,,
\]
then the map $v:=\varphi-z_1$ is a solution of the translated equation \eqref{perturbedEq} with $z=z_1$, so that
by the uniqueness part above there holds:
$\varphi = h(h_0,z_1,\cdot )|_{[0,T_1]}$.
Note that $\frac{\d f}{\d t}\in C([0,T_1];V_{\beta-2})$, so that by classical theory of parabolic equations, we have indeed $z_1\in C([0,T_1];V_{\beta})$.

We have now all at hand to prove Theorem \ref{thm:main}.
The proof follows standard arguments, which are detailed for the sake of completeness (the key property being Proposition \ref{mainProp}).
\begin{proof}[Proof of Theorem \ref{thm:main}]
Fix $t_*>0$, $s\in(0,t_*)$
and take $\mathfrak{H}, \mathfrak{Z}$ as in Proposition \ref{mainProp}, with $t_*$ replaced by $t_*-s$. Since the interior of $\mathfrak H$ is nonempty, we can take $h_1\in\interior\mathfrak{H}$.
By the controlability property \eqref{controlability}, there exists $z_1\in C([0,s];V_\beta)$ such that $h(h_0,z_1,\cdot )$ is defined on $[0,s]$ and $h(h_0,z_1,s)=h_1$.
Using in addition Lemma \ref{lem:c_dep_result}, we see that there exists a neighbourhood $\mathcal{V}_1$ of $z_1$ in $C([0,s];V_\beta)$, such that 
\[\forall z\in\mathcal{V}_1\,,\enskip h(h_0,z,s)\in{\mathfrak{H}}\,.\]
Since $\ker\phi ^*=\{0\}$, then $\phi$ has dense range in $V_\beta$ and the process 
$Z(t)=\int_0^tS(t-\sigma )\d w_\phi(\sigma )$, $t\geq 0$,
is non-degenerate.
Therefore,
\begin{align}
\label{positive_1}
&p_0:=\P\circ Z|_{[0,s]}^{-1}(\mathcal{V}_1)>0\,,\\
\intertext{and similarly}
\label{positive_2}
&p_1:=\P\circ Z|_{[0,t_*-s]}^{-1}(\mathfrak{Z})>0\,.
\end{align}
Now, define the extended state space
$\mathfrak X =V_\beta \cup\{\vartriangle\}$ where the terminal state $\vartriangle$ is an isolated point,
and extend the process $X_{t,h_0}(\omega ):=h(h_0,t,Z(\omega ))$ on $\mathfrak X$, by achieving $\vartriangle$ if and only if $t\geq \tau ^\beta (h_0,Z(\omega ))$.
By standard arguments (see e.g.\ \cite{romito2014uniqueness} and references therein),
the family of probability measures $\left(\P_x\equiv\P\circ X_{\cdot ,\,x}^{-1}\right)_{x\in \mathfrak X }$ on $\bar W:=C([0,\infty);\mathfrak X)$, the space of trajectories equipped with the $\sigma $-algebra corresponding to Borelian sets, is Markovian.
Letting $A:=\{w\in\bar W:\tau (w)\geq s\}$, we have
\begin{equation}\label{this}
\P_x\left(A\cap\{w:w(t_*)=\vartriangle\}\right)=\int_{A}\P_{w'(s)}(w:w(t_*-s)=\vartriangle)\P_x(\d w')\,.
\end{equation}
Denote by $P(x,t;\cdot )$ the associated transition probabilities, namely $\P_x(w:w(t)\in\Gamma)$ where $\Gamma\subset \mathfrak X $ is Borelian, and by
$\pi_{s}:\bar W\to \mathfrak X $, $w\mapsto w(s)$. Then
\eqref{this} implies
\[\P_x\left(\tau \leq t_*\right)\geq\int\limits_{\mathfrak H\cap\pi_{s}(A)}P(t_*-s,\xi ;\{\vartriangle\})P\left(s,x;\d \xi \right)\geq p_1\int\limits_{\mathfrak H\cap \pi_s(A)}P\left(h_0,s;\d \xi \right)\,,\]
where we have used \eqref{positive_2} to bound $P(t_*-s,\xi,\{\vartriangle\})$ independently of $\xi \in\mathfrak H$.
Using in addition \eqref{positive_1}, we obtain $\P_x(\tau \leq t_*)>p_1p_0$ which is positive.
Furthermore, using the equivalence of formulations \eqref{eq:weak_u} and \eqref{eq:mild_h} (together with Remark \ref{rem:equiv_A_Delta}), then
Theorem \ref{thm:main} is proved.
\end{proof}

\begin{closing}
Ineluctability of blow-up for 1-corotational solutions of \eqref{SHMFprime} remains an open problem. Let $(\P_x)_{x\in \mathfrak X}$ be the Markovian family on $(\bar W,\mathscr B(\bar W))$ defined in the proof of Theorem \ref{thm:main}.
The following observation is made in \cite[sec.~5]{romito2014uniqueness} (note that we could also let $\mathfrak X$ be any Polish space with additional isolated point $\{\vartriangle\}$):
\begin{adjustwidth}{1em}{}
\hspace{1em} For $w\in\bar W$ denote by $\tau(w) :=\inf\{t\geq 0;w(t)=\vartriangle\}$.
Assume that there exist $T>0$, and an open set $B_0\subset\mathfrak X$, such that:
\begin{hyp}[Uniform lower bound]\label{romito1}
There exists a constant $p_0$, independent of $x\in B_0$ with $\P_x(\tau \leq T)\geq p_0$;
\end{hyp}
\begin{hyp}[Conditional recurrence]\label{romito2}
For all $x\in\mathfrak X$, $\P_x(\sigma=\infty\text{ and }\tau =\infty)=0$, where for $w\in\bar W$, $\sigma(w)$ is defined as $\inf\{t\geq 0,w(t)\in B_0\}.$
\end{hyp}
\noindent Then for each $x\in\mathfrak X$
\[\P_x(\tau <\infty)=1\,.\]
\end{adjustwidth}

Taking $B_0:=\interior{\mathfrak H}$, where $\mathfrak H$ is as in Proposition \ref{mainProp}, then Condition \ref{romito1} has already been checked in \eqref{positive}: it suffices to let $p_0:=\P\circ Z|_{[0,T]}^{-1}(\mathfrak Z)$, where $\mathfrak Z$ is as in Proposition \ref{mainProp}.
However Condition \ref{romito2}, i.e.\ the conditional recurrence for the pre-blow-up set $\interior{\mathfrak H}$, seems difficult to check, because it relates large time behaviour of solutions of \eqref{SHMFprime}.
A natural idea would be to replace first $\mathfrak H$ by some neighbourhood $\mathcal V$ of $0$ in $C^1([0,1])$, say, and then to bound below the probability to reach $\mathfrak H$ from $\mathcal V$.
In the deterministic case, such stability results are for instance obtained in \cite{kung1989heat} or \cite{carbou1997comportement} for the full LLG equation, and rely on the energy estimate $E(t)-E(0)+\iint_{[0,t]\times\D}|u\times\Delta u|^2=0$, which gives uniform bounds in $t>0$.
The main difficulty here is that the counterpart of the above identity writes:
\begin{equation}\label{energy-id}
E(t)-E_0 +\iint_{[0,t]\times\D}|u\times\Delta u|^2\d t= C_\phi t + M_\phi (t)
\end{equation}
$M_\phi $ being a martingale, and $C_\phi $ a positive constant, but note that \eqref{energy-id} is not sufficient to obtain uniform boundedness of $\frac1t\E\int_0^tE(s)\d s$.
\end{closing}
\begin{closing}
As already mentioned in the introduction, it is not expected that the pre-blow-up sets remain open if we release the 1-corotational symmetry assumption. Consider maps with \textit{two degrees of freedom}:
$u_{g,h}(t,x):=(\cos g\sin h,\sin g\sin h,\cos h)$
where $x=(r\cos\theta ,r\sin\theta )$, $g=g(t,r,\theta )$ and $h=h(t,r,\theta )$.
Putting $u_{g,h}$ in \eqref{SHMF},
then we obtain the following parabolic system:
\begin{equation}\label{eq:2_degrees}
\left\{\begin{aligned}
&\d g=\left(\partial _{rr}g+\frac{\partial _rg}{r}+\frac{\partial _{\theta \theta }g}{r^2}+\frac{2}{\tan h}\Big(\partial _rg\partial _rh+\frac{\partial _\theta g\partial _\theta h}{r^2}\Big)\right)\d t+\frac{1}{\sin h}\circ \d w_1\\
&\d h=\left(\partial_{rr} h+\frac{\partial _rh}{r}+\frac{\partial _{\theta \theta }h}{r^2}-\Big(\partial _rg^2+\frac{\partial _\theta g^2}{r^2}\Big)\sin2h/2\right)\d t+\d w_2\,,
\end{aligned}\right.
\end{equation}
where $w_1(t,r,\theta ),w_2(t,r,\theta )$ are independent.

The above conjecture gives some indication that blow-up phenomenon should not happen for \eqref{eq:2_degrees},
even if $u$ is 1-equivariant, that is $g(t,r,\theta )=\theta + \tilde g(t,r) $ and where, in order to preserve this symmetry, we would take $w_j=w_j(t,r)$ for $j=1,2$.
Non-constant $\tilde g(t,r) $ are shown in \cite{merle2011blow} to stabilize the solutions of the Heisenberg equation, which is related to the fact that the gyromagnetic term $u\times\Delta u$ makes the solution turn around the vertical axis $\mathbf k\equiv(0,0,1)$.
This necessary extra degree of freedom also appears when taking the full noise term $\frac{ \Theta_u }{\sin h }\circ\d w_1+\Phi_u\circ \d w_2$ in the equation.
For this reason, we believe that finite-time blow-up for general solutions of \eqref{SHMF} is a zero-probability event.
\end{closing}

\appendix
\renewcommand\theequation{A.\arabic{equation}}
\section{Complements in the proof of Theorem \ref{thm:loc_solv}}
\label{appendix}
\subsection{Self-adjointness of $A$}
\label{app1}
Let $(A_1,D(A_1))$ be the unbounded linear operator defined by the operation $\partial _{rr}+(\frac1r \partial _r-\frac{1}{r^2})$ on the domain
\[
D(A_1):=\{f\in C^\infty(\I;\R)\,,\, f(0)=f(1)=0\}\,.
\]
For $f\in D(A_1),$ the term ``$(\frac1r \partial _r-\frac{1}{r^2})f$'' which should be understood in distributional sense over the interval $I\setminus\{0\}$, is well defined by the fact that
for $r$ in $(0,1]:$
\[
\left(\frac{\partial _rf(r)}{r} -\frac{f(r)}{r^2}\right)= \frac{1}{r}f'(0) -\frac1r\left(\lim_{r\to0}\frac{f(r)}{r}\right) + g(r)\equiv g(r)\,,
\]
where from Taylor Formula the remainder $g$ belongs to $H.$

By \cite[Theorem X.39]{reed1975methods}, since $A_1$ is symmetric (this follows by \eqref{ipp}),
and since linear combinations of eigenvalues are dense in $H$ (the Bessel functions are smooth),
it is \emph{essentially self-adjoint}.

Now, let $f_n,n\geq 0,$ $f,g$ in $H$ such that $f_n\to f$ and $Af_n\to g.$
Note that $Af_n$ also converges to $Af$ as a distribution in $\mathscr D'(0,1),$ hence $Af\equiv g\in H.$
Owing to \eqref{Delta_F}-\eqref{plugg_ansatz}, it follows that we have also
\[
\int_\I \left[(\partial _{rr}f)^2+\left(\frac1r\partial _rf-\frac{f}{r^2}\right)^2\right]r\d r<\infty\,,
\]
hence $f\in D(A)$ (continuity follows from Remark \ref{rem:Sobolev_embedding}).
This shows that $(\bar{A_1},D(\bar{A_1}))\subset (A,D(A)).$

Conversely, if $f$ belongs to $D(A),$ consider the sequence $f_n:=\tilde f*\rho _{n}|_\I,$
where $\tilde f$ is as in \eqref{1-fattening}, and $\rho _n\in C^\infty_0(-1,1)$ is an even approximation of the Dirac delta.
From the same computations as in \eqref{id_theta}-\eqref{id_theta2}, we have 
\[
f_n(0)\equiv\int_{[-1,1]}\tilde f(r')\rho _n(-r')r'\d r'=0\,,\quad \text{and similarly}\quad f_n(1)=0\,,
\]
(this holds because $\rho_n $ is even, whereas we extend $f$ in a skew-symmetric way)
hence $f_n\in D(A).$
Moreover $f_n\to f$ in $H,$ and since $Af_n=\tilde {Af}*\rho _n$ it follows that we have also 
\[
Af_n\to Af\quad \text{in}\enskip  H\,.
\]
Hence the opposite inclusion is also true, so that $(A,D(A))\equiv(\bar{A_1},D(\bar {A_1}))$ is self-adjoint.
\hfill\qed

\subsection{Higher regularity}\label{app2}
\paragraph{Local solvability.}
Take $2<\beta< 4, \beta '>\beta -1,$ let $h_0\in V_{\beta}$ and assume that $\phi\in \mathbb{L}_2(H;V_{\beta'} )$. By the same argument as above, we can fix $z\in C([0,T_*];V_{\beta})$ with $z(0)=0$, and argue pathwise.
Denote by $(h,\tau^{\beta-2})$, the maximal solution obtained in Section \ref{sec:proof_main_thm}, which therefore belongs to $C([0,\tau^{\beta-2});V_\beta)$. We aim to find an \textit{a priori} bound on $\|h\|_{T,\beta }$ guaranteeing existence during a positive time.
Write for $0\leq t<\tau^{\beta-2}$:
\begin{equation}\label{eq:Gamma_2}
h(t)=S(t)h_0+\int_0^t(-A)^{\delta } S(t-s)\left[(-A)^{-1-\delta }(-Ab_{h})\right]\d s+z(t)\,.
\end{equation}
where $\delta :=(\beta -2)/2\in(0,1)$,
and using \eqref{P1}, we obtain the bound
$|h(t)|_\beta\leq |S(t)h_0|_\beta+\int_0^t(t-s)^{-\delta }|Ab_{h}|_{H} \d s+|z(t)|_\beta $,
provided all terms are finite.
Therefore, there remains to evaluate the term
$|Ab_{v}|_{H}$.
Direct computations lead to 
\[\begin{aligned}
Ab(r,h)&=\frac{1-\cos2h}{r^2}\partial_{rr}h\\
&\hspace{2em}+ \frac{1-\cos2h}{r^3}\partial_rh-\frac{6h-3\sin2h}{2r^4}\\
&\hspace{4em}+ \frac{2\sin2h}{r^2}(\partial_rh)^2 +\frac{6h-3\sin2h}{r^4} -\frac{4(1-\cos2h)}{r^3}\partial_rh\enskip,
\end{aligned}\]
where, due to compensations, each line of the right hand side must be treated separately.
Using the triangle inequality, we write for $h\in V_\beta$,
$|Ab(r,h)|_{H}\leq I+II+III$, and deal with each term separately.
For the sake of clarity, from now until the end of the proof, we use the notation $T_1(h)\lesssim T_2(h)$ if two terms involving $h\in V_\beta$ are comparable up to a multiplicative constant that does not depend on $h$.

In the sequel, we fix an arbitrary $\epsilon>0$.
Using the bound 
$|G(x)|\leq c|x|^2,\enskip x\in \R$, where $G:x\in\R\mapsto 1-\cos(2x)$,
Remark \ref{rem:equiv_A_Delta}, and Lemma \ref{lem:alpha}--\ref{lem:alpha:item_i} in the case $\nu=1$, $p=\infty$,
the first term satisfies
$I\lesssim|\frac{h}{r}|_{L^{\infty}}^2|\partial_{rr}h|_{H}\lesssim|h|_{{2+\epsilon}}^2|h|_{2}$,
whereas for the second term we have:
\[
\begin{aligned}
II&\lesssim\big|\frac{h^2}{r^2}(\frac{\partial_rh}{r}-\frac{h}{r^2})\big|_{H}+\big|\frac{1-\cos2h-2h^2}{r^3}\partial_rh-\frac32\big(\frac{2h-\sin2h-(4/3)h^3}{r^4}\big)\big|_{H}\\ 
&\enskip \enskip =II_1+II_2\,.
\end{aligned}
\]
Using Lemma \ref{lem:alpha}--\ref{lem:alpha:item_i} with $\nu=1$, $p=\infty$, and Remark \ref{rem:equiv_A_Delta}, there holds
$II_1\lesssim |h/r|_{L^{\infty}}^2|\partial_rh/r-h/r^2|_{H}\lesssim|h|_{{2+\epsilon}}^2|h|_{2}$.
Moreover, by the classical inequalities
$|1-\cos2x -2x^2|\leq cx^4$, $|2x-\sin2x-(4/3)x^3|\leq c|x|^5$ for $x\in \R$,
H\"{o}lder inequality, and Lemma \ref{lem:alpha}--\ref{lem:alpha:item_i} with $(\nu,p)=(3/4,40/3)$, and then \ref{lem:alpha:item_ii} with $p=5$, (resp.\ \ref{lem:alpha:item_i} with $(\nu,p)=(4/5,10)$), the following bound is obtained:
\[ II_2\lesssim\big|\frac{h^4}{r^3}\partial_rh\big|_{H}+\big|\frac{h^5}{r^4}\big|_{H}
\lesssim\big|\frac{h}{r^{3/4}}\big|_{L^{40/3}_{r\d r}}^4\big|\partial_rh\big|_{L^5_{r\d r}}+\big|\frac{h}{r^{4/5}}\big|_{L^{10}_{r\d r}}^5
\lesssim \big|h\big|_{{8/5+\epsilon}}^5\,.
\]

The bound on $III$ is obtained in a similar way.
We write that $III\leq III_1+III_2$, with
\[\begin{aligned}
III_2&=\big|\frac{2\sin2h-4h}{r^2}(\partial_rh)^2 +3\big(\frac{2h-\sin2h-(4/3)h^3}{r^4}\big)-4\big(\frac{1-\cos2h-2h^2}{r^3}\big)\partial_rh\big|_{H}\\
&\lesssim \big|\frac{h}{r^{2/3}}\big|_{L^{30}_{r\d r}}^3\big|\partial_rh\big|_{L^5_{r\d r}}^2+ \big|\frac{h}{r^{4/5}}\big|_{L^{10}_{r\d r}}^5+\big|\frac{h}{r^{3/4}}\big|_{L^{40/3}_{r\d r}}^4\big|\partial_rh\big|_{L^5_{r\d r}}\enskip
\end{aligned}\]
which is bounded by $c|h|_{{8/5+\epsilon}}^5$, by the Sobolev embeddings of Remark \ref{rem:Sobolev_embedding}.
For the first term $III_1=|(h/r)\partial_rh(\partial_rh/r-h/r^2)+h^2/r^2(h/r^2-\partial_rh/r)|_{H},$ we use Remark \ref{rem:equiv_A_Delta} and Lemma \ref{lem:alpha}--\ref{lem:alpha:item_i} with $(\nu,p)=(1,\infty)$. We finally get:
\[III_1\lesssim\big|\frac{h}{r}\big|_{L^{\infty}}\Big|\frac{\partial_rh}{r}-\frac{h}{r^2}\Big|_{H}\left(\big|\partial_rh\big|_{L^{\infty}}+\big|\frac{h}{r}\big|_{L^{\infty}}\right)\lesssim |h|_{2}|h|_{{2+\epsilon}}^2\,.
\]

Going back to \eqref{eq:Gamma_2}, and fixing $\epsilon>0$, we see that for some
constant $c_\epsilon>0$:
\begin{equation}\label{eq:Qh}
|h(t)|_{\beta}\leq c|h_0|_{\beta}+c_\epsilon\int_0^t(t-s)^{-\delta }g(s)\d s +\|z\|_{\tau^{\beta-2},\,\beta}\,,\quad t\in[0,\tau^{\beta-2})\,,
\end{equation}
where we let 
\begin{equation}\label{nota:g}
g(s):=|h(s)|_{{8/5}+\epsilon}^5+|h(s)|_{{2+\epsilon}}^2|h(s)|_{2}\,.
\end{equation}
By a classical generalization of Gronwall Lemma, \eqref{eq:Qh} implies existence in $V_\beta $ for some positive time $0<\tau^\beta \leq\tau^{\beta -2}$. By Remark \ref{rem:equiv_A_Delta}, existence in $\H^\beta $ follows.
\hfill\qed

\paragraph{Propagation of regularity.}
Since the integrand $g(s)$ defined in \eqref{nota:g} does not depend on $|h(s)|_{\beta }$ as soon as $\beta \geq 2+\epsilon ,$ we see that $h$ can always be extended continuously in $V_\beta $ after $t$ provided $|h(t)|_{2+\epsilon }<\infty.$
Therefore:
\begin{equation}\label{eq_tau}
\tau ^{2+\epsilon }=\tau ^{\beta }\enskip \text{for every}\enskip 2+\epsilon \leq \beta<4\, ,
\end{equation}
and every $\epsilon >0$.
Recalling that $|u_h|_{\mathscr H^\beta },|h|_{V_\beta }$ are equivalent quantities (see Remark \ref{rem:equiv_A_Delta}),
this finishes the proof of Theorem \ref{thm:loc_solv}.
\hfill\qed
\renewcommand\theequation{B.\arabic{equation}}
\section{Proof of Lemma \ref{lem:c_dep_result}: continuous dependence of the solution $h(h_0,z)$ with respect to its arguments.}
The following proof is adapted from that of \cite{de2002effect,de2003stochastic}.
In the sequel, we fix $\beta\in(4/3,2]$, $h_0\in V_\beta $ and $z\in C([0,T];V_\beta )$. For $R,T>0$, we denote by $\mathbb{B}_T^R$ (resp.\ $B^R$) the ball of radius $R$, centered at $z$ in $C([0,T];V_{\beta})$ (resp.\ $h_0$ in $V_\beta$).
If $(h_1,\zeta )\in V_\beta\times C([0,T];V_{\beta}),$ we will denote by
$v(h_1,\zeta ,\cdot )$ the corresponding (maximal) mild solution of \eqref{perturbedEq}, obtained by reiteration of the fixed point argument for $\Gamma_{h_1,\zeta ,T}$
(see Section \ref{sec:proof_main_thm}). We also denote by $\tau (h_1,\zeta )$ its existence time.
\begin{proof}[Proof of Lemma \ref{lem:c_dep_result}]
Assume that $\tau (h_0,z)>T$, and let
\begin{equation}\label{def_R}
R:=\|v(h_0,z)\|_{T,\beta }\vee\|z\|_{T,\beta }+1\,,
\end{equation}
define $T_*(R)$ as in \eqref{T_star}, and set $N:=\lfloor T/T_*\rfloor$.
We prove the result by induction.
For each $k\in\{1,\dots,N\}$ denote by $(H_k)$ the sentence:

\begin{Hk}
``There exists $\delta_k>0$, such that
if $(h_0,z)\in B^{\delta_k}\times\mathbb{B}_T^{\delta_k}$, then:
$\tau(h_0,z)>kT_*$, and the map $(h_0,z)\in B^{\delta_k}\times\mathbb{B}_{kT_*}^{\delta_k}\mapsto v(h_0,z,kT_*)$ is continuous.''
\end{Hk}

The case $k=1$ has been proved in Section \ref{sec:proof_main_thm}: 
it suffices to take $\delta _1>0$ depending on $R$ only, so that \eqref{relation_fix} holds for all $(h_1,\zeta )\in B^{R}\times\mathbb B^{R}_{T_*}$.

\item[\indent\textit{Inductive step.}]
Let $k\geq 1$ and assume $(H_\ell )_{0\leq\ell \leq k}$.
In particular $(H_k)$ implies the existence of $\delta>0$, such that
$|v(h_1,\zeta ,kT_*)-v(h_0,z,kT_*)|_{\beta}<\delta _1$ for every $(h_1,\zeta )\in B^{\delta }\times\mathbb{B}^{\delta}_{kT_*}$.
For $t\in[0,T_*]$, denoting by $x (t):=z (t+kT_*)-S(t)z (kT_*)$, by $\xi (t):=\zeta (t+kT_*)-S(t)\zeta (kT_*)$, and assuming without loss of generality that $\delta <\delta _1/2$,
we have $\|\xi-x \|_{T_*,\beta }<\delta _1$. 
By $(H_1)$, this implies that $v(h_1,\zeta ,\cdot )$ is at least defined up to $(k+1)T^*$. Moreover, by uniqueness:
\begin{equation}\label{expression}
v\big(h_1,\zeta ,(k+1)T_*\big)=v\big(v(h_1,\zeta ,kT_*),\xi ,T_*\big)\,.
\end{equation}
Still by $(H_1)$, \eqref{expression} defines a continuous map with respect to $(h_1,\zeta )\in B^{\delta_k\wedge \delta }\times\mathbb{B}_{(k+1)T_*}^{\delta_k\wedge \delta }$. This proves $(H_{k+1})$, letting $\delta _{k+1}:=\delta_k\wedge\delta$.

In particular, $(H_N)$ is true, which implies the proposition when $\beta \in (4/3,2]$.
Higher regularity is standard.
\end{proof}
\renewcommand\theequation{C.\arabic{equation}}
\section{Proof of the comparison principle}
\label{app3}
For $\J:=[0,r_1]\subset \I$, we denote the parabolic boundary by $\Sigma_\kappa :=\{0\}\times \J\cup [0,\kappa )\times\partial \J$.
To avoid cumbersome computations, when $f\in H$ we denote by 
$\int_{\J}f:=\int_{\J}f(r)r\d r\,$, and fixing $z$ as in \eqref{ass_z} we write
\[q_f(t,r):=f(t,r)+p\big(z(t,r)+f(t,r)\big)\,,\quad (t,r)\in[0,\kappa )\times \J.\]

Take now $0<T<\kappa\,$, and let $t\mapsto\zeta(t,\cdot)\in C([0,T)\times[0,r_1])$ be a non-negative map such that $\zeta(t,r)$ vanishes for $(t,r)\in\Sigma_T$.
Using \ref{ineq:df_pro:comparison_principle} and \ref{ineq:dg_pro:comparison_principle}, we obtain
\begin{equation}\label{ineq:f_g}
-\int_0^t\int_{\J}(f-g)\partial _t\zeta \leq-\int_0^t\int_{\J}\partial _r(f-g)\partial _r\zeta +\frac{(q_f-q_g)\zeta}{r^2}\,.
\end{equation}

Recall that if $\varphi \in V_2$, and $\psi \in V_1$, there holds the integration by parts formula \eqref{ipp}.
Due to \eqref{ass_z}-\eqref{ass}, and because of $f,g\in C([0,T];V_1)$, then
the right hand side of \eqref{ineq:f_g} is bounded by $c\|f-g\|_{T,1}\|\zeta \|_{T,1}$. By density \eqref{ineq:f_g} can thus be extended to the larger class of test functions 
\[
\mathscr T:=\left\{\zeta:[0,T]\times \J\to\R_+\,,\,\zeta|_{\Sigma_T}=0\text{ and } \|\zeta \|_{T,1}+\|\partial _t\zeta \|_{T,0}<\infty\right\}\,.
\]
Denote by $[x]_+:=\max\{x,0\}$, and define $\zeta(t,r):=[f-g]_+(t,r)$. The fact that $f,g\in C^1([0,T];H)$ implies
\begin{equation}\label{derivee_temps}
\frac{\d}{\d t} \int_{\J}[f-g]_+^2=2\int_{\J}\partial_t(f-g)\zeta \,,
\end{equation}
which is summable on $[0,T]$.
Noticing furthermore that $\zeta \in\mathscr T$ (note that $f\in V_1\Rightarrow [f]_+\in V_1$),
applying \eqref{ineq:f_g} to $\zeta $, \eqref{derivee_temps} and integrating by parts gives:
\begin{equation}\label{positive_term}
\begin{aligned}
\frac12\int_{\J}[f(t)-g(t)]_+^2&\leq-\iint_{[0,t]\times \J}\mathds{1}_{f\geq g}\left(\partial_r(f-g)\right)^2+\frac{[f-g]_+(q_f-q_g)}{r^2}\\
&\leq -\iint_{[0,t]\times \J}\frac{[f-g]_+(q_f-q_g)}{r^2}\enskip,
\end{aligned}
\end{equation}
where we have used the fact that the weak derivative of $x\in\R\mapsto [x]_+,$ is the map
$x\in\R\mapsto\mathds{1}_{\mathbb{R_+}}(x)$.
By \eqref{ass_z}-\eqref{ass}, and since $\beta>1$ (implying the uniform continuity of $z,f,g$ on compacts, see Remark \ref{rem:Sobolev_embedding}),
we can find $\varepsilon(t,r)$ depending on $p''(0),f,g,$ such that $\varepsilon(t,r)\to0$ as $r\to0$, uniformly in $t\in[0,T]$, and such that
$q_f(t,r)-q_g(t,r)=(1+p'(0)+\varepsilon(t,r))(f(t,r)-g(t,r))$.
Since $p'(0)>-1$ and $f|_{\Sigma}\leq g|_{\Sigma}$, this yields the existence of
$\bar r=\bar r(T)$ such that:
\begin{equation}\label{positive}
[f-g]_+(q_f-q_g)\geq0\enskip \text{for a.e.}\enskip(t,r)\in[0,T]\times[0,\bar r]\,.
\end{equation}
Finally we write for all $t\in[0,T]$:
\[\int_{\J}[f(t)-g(t)]_+^2\leq-\iint_{[0,t]\times[0,\bar r]}\frac{[f-g]_+(q_f-q_g)}{r^2}+\frac{1}{2\bar r^2}\iint_{[0,t]\times[\bar r,1]}[f-g]_+|q_f-q_g|\,,\]
which by \eqref{positive} and \eqref{ass}, is bounded by
$\frac{K}{\bar r^2}\iint_{[0,t]\times \J}[f-g]_+^2.$
We finally obtain
$|{[f-g]_+}|_H^2(t)\leq\frac{K}{\bar r^2}\int_0^{t}|{[f-g]_+}|_H^2(s)\d s$ for $t\in[0,T]$, and $[f-g]_+|_{[0,T]\times \J}\equiv0$ follows by Gronwall Lemma.
Reiterating on every subinterval $[0,T]\subset[0,\kappa)$ gives ${f}\leq {g}$ on $[0,\kappa)$.\hfill\qed

\section*{Acknowledgements}
AH wish to acknowledge Anne De Bouard and Fran\c{c}ois Alouges for precious guidance.
Arnaud Debussche and Marco Romito are also warmly thanked for suggesting possible improvements. Partial funding of this research through the the ANR projects Micro-MANIP (ANR-08-BLAN-0199) and STOSYMAP (ANR-2011-BS01-015-03) is gratefully acknowledged.
The author is also grateful to the anonymous referees, especially for having revealed some gaps in the first version of this paper.

\bibliographystyle{elsarticle-num}

\end{document}